\documentclass[a4paper,12pt]{amsart}
\usepackage{mathrsfs}
\usepackage{url}
\usepackage{amssymb}
\usepackage{latexsym}
\usepackage{amsfonts}
\usepackage{amsmath}
\usepackage{eucal}
\usepackage{bm}
\usepackage{bbm}
\usepackage{graphicx}
\usepackage[english]{varioref}
\usepackage[nice]{nicefrac}
\usepackage[all]{xy}
\usepackage{amsthm}

\newcommand{\supp}{\text {\rm supp}}

\newcommand{\Hom}{{\rm Hom}}

\def\i{^{-1}}
\def\ge{\geqslant}
\def\le{\leqslant}
\def\<{\langle}
\def\>{\rangle}

\def\d{\text{d}}
\def\Dr{\text{D}}
\def\gr{\text{grad}}
\def\Exp{\text{exp}}
\def\Lim{\text{Lim}}
\def\sup{\text{sup}}
\def\tr{\text{Tr}}
\def\K{\text{K}}
\def\Aut{\text{Aut}}

\def\a{\alpha}

\def\d{\delta}
\def\D{\Delta}

\def\e{\epsilon}

\def\o{\omega}

\def\s{\sigma}
\def\t{\tau}
\def\th{\theta}

\def\l{\lambda}

\def\Om{\Omega}

\def\ZZ{\mathbb Z}
\def\NN{\mathbb N}
\def\QQ{\mathbb Q}

\def\RR{\mathbb R}
\def\CC{\mathbb C}

\def\ca{\mathcal A}

\def\cd{\mathcal D}

\def\co{\mathcal O}
\def\cp{\mathcal P}

\def\tM{\tilde M}
\def\tJ{\tilde J}
\def\tH{\tilde H}

\def\tu{\tilde u}
\def\tx{\tilde x}
\def\ty{\tilde y}

\def\tW{\tilde W}
\def\tw{\tilde w}

\def\fA{\mathfrak A}
\def\fC{\mathfrak C}
\def\fH{\mathfrak H}

\theoremstyle{plain}
\newtheorem{thm}{Theorem}[section]
\newtheorem*{thm*}{Theorem}
 \newtheorem{prop}[thm]{Proposition}
 \newtheorem{lem}[thm]{Lemma}
 \newtheorem{cor}[thm]{Corollary}

\theoremstyle{definition}

\theoremstyle{remark}

\newtheorem*{rmk}{Remark}
\newtheorem*{claim*}{Claim}

\begin{document}
\author{Xuhua He}
\address{Department of Mathematics, The Hong Kong University of Science and Technology, Clear Water Bay, Kowloon, Hong Kong}
\email{maxhhe@ust.hk}
\author{Sian Nie}
\address{Institute of Mathematics, Chinese Academy of Sciences, Beijing, 100190, China}
\email{niesian@gmail.com}
\title[]{Minimal length elements of extended affine Coxeter groups, II}

\begin{abstract}
Let $W$ be an extended affine Weyl group. We prove that minimal length elements $w_{\co}$ of any conjugacy class $\co$ of $W$ satisfy some special properties, generalizing results of Geck and Pfeiffer \cite{GP} on finite Weyl groups. We then introduce the ``class polynomials'' for affine Hecke algebra $H$ and prove that $T_{w_\co}$, where $\co$ runs over all the conjugacy classes of $W$, forms a basis of the cocenter $H/[H, H]$. We also classify the conjugacy classes satisfying a generalization of Lusztig's conjecture \cite{L4}.
\end{abstract}

\maketitle

\section*{Introduction}

\subsection{} Let W be a finite Weyl group and $\co$ be a conjugacy class of $W$. In \cite{GP} and \cite{GP00}, Geck and Pfeiffer proved the following remarkable properties:

(1) For any $w \in \co$, there exists a sequence of conjugations by simple reflections which reduces $w$ to a minimal length element in $\co$ and the lengths of the elements in the sequence weakly decrease;

(2) If $w, w'$ are both of minimal length in $\co$, then they are strongly conjugate.


Such properties play an important role in the study of finite Hecke algebras. They lead to the definition and determination of ``character tables'' for finite Hecke algebras, which are in analogy to character table of finite groups. They also play a role in the study of Deligne-Lusztig varieties (see, e.g. \cite{OR}, \cite{BR} and \cite{HL}) and in the study of links between conjugacy classes in finite Weyl groups and unipotent conjugacy classes in reductive groups (see \cite{L?}).

\subsection{} The main purpose of this paper is to establish similar properties for affine Weyl groups. Such properties were first studied by the first author for some affine Weyl groups of classical type in \cite{He2} via a case-by-case analysis.

The method we use here is quite different from loc.cit. We represent elements in the conjugacy class containing a given element $\tw$ by alcoves. The key step is Proposition \ref{red1}, which reduces the study of arbitrary alcove to an alcove that intersects the affine subspace associated to the Newton point of $\tw$. We then use the ``partial conjugation'' method introduced in \cite{He1} to reduce to a similar problem for a finite Coxeter group, which is proved in  \cite{GP} and \cite{GP00} via a case-by-case analysis and later in \cite{HN} by a case-free argument. This leads to a case-free proof which works for all cases, including the exceptional affine Weyl groups, which seems very difficult via the approach in \cite{He2}. 

We'll discuss in the next section some applications of these properties to affine Hecke algebras. They also play a key role in the study of affine Deligne-Lusztig varieties and a proof of the main conjecture in \cite{GHKR}. See \cite{H99} and \cite{GHN}.

\subsection{} The affine Hecke algebra $H$ is a free $\ZZ[q, q \i]$-module with basis $T_w$ for $w$ runs over elements in $W$. By density theorem and trace Paley-Wiener theorem \cite{K}, if $q$ is a power of a prime, then the trace function gives a natural bijection from the dual space of the cocenter $H/[H, H]$ to the space of Grothedieck group of representations of $H$.

We construct a standard basis for the cocenter $H/[H, H]$. It consists of elements $T_{w_\co}$, where $\co$ runs over all the conjugacy classes of $W$ and $w_\co$ is a minimal length representative in $\co$. The special properties on $W$ implies that

(a) The image of $T_{w_\co}$ in the cocenter $H$ does not depend on the choice of minimal length representative $w_\co$.

(b) For any $w \in W$, the image of $T_w$ in the cocenter of $H$ is a linear combination of $T_{w_\co}$ and the coefficients are the ``class polynomials'' $f_{w, \co}$.

The fact that the images of various $T_{w_\co}$ in the cocenter of $H$ are linearly independent follows from the density theorem for affine Weyl group, which will be proved in section 6.

\subsection{} Among all the conjugacy classes of $W$, the straight conjugacy classes play an intriguing role. By definition, a conjugacy class is called straight if it contains a straight element $w$, i.e., $\ell(w^n)=n \ell(w)$ for all $n \ge 0$. The minimal length elements in a straight conjugacy class are just the straight elements it contains. The notion of straight element/conjugacy class was first introduced by Krammer in \cite{Kr} in the study of conjugacy problem.

It was observed by the first author in \cite{He2} that the straight conjugacy classes have some geometric meaning. Namely, there is a natural bijection between the set of $\s$-conjugacy classes of a $p$-adic group and the set of straight conjugacy classes of the corresponding affine Weyl group $W$. There is no known counterpart for finite Weyl groups.

In this paper, we also study in detail some special properties of minimal length elements in a straight conjugacy class. We give some geometric and algebraic criteria for straight conjugacy classes and show that any two minimal length elements in a straight conjugacy class are conjugate by cyclic shift. We also show that two straight elements are conjugate to each other if and only if their Newton points are the same.

\subsection{} An elliptic conjugacy class $\co$ in a finite Coxeter group satisfies the following remarkable properties:

(1) Any two minimal length elements in $\co$ are conjugate by cyclic shift.

(2) For any minimal length element $\tw$ of $\co$, the centralizer of $\tw$ is obtained via cyclic shift in the sense of Lusztig \cite{L4}.

Property (1) was due to Geck and Pfeiffer \cite{GP00} and the first author \cite{He1}. Property (2) was first conjectured and verified by Lusztig \cite[1.2]{L4} for untwisted classical groups and proved in general in \cite{HN}.

Such properties play an important role in the study of Deligne-Lusztig varieties and representations of finite groups of Lie type. Similar properties for affine Weyl group also play an essential role in the study of affine Deligne-Lusztig varieties.

In this paper, we also classify the conjugacy classes of finite and affine Weyl groups that satisfy properties (1) and (2) above. In particular, we show that any conjugacy class in affine Weyl group whose finite part is elliptic in the corresponding finite Weyl group satisfies both properties.

\section{Preliminary}

\subsection{}\label{Cox} Let $S$ be a finite set and $(m_{s t})_{s, t \in S}$ be a matrix with entries in $\mathbb N \cup \{\infty\}$ such that $m_{s s}=1$ and $m_{s t}=m_{t s} \ge 2$ for all $s \neq t$. Let $W$ be a group generated by $S$ with relations $(s t)^{m_{s t}}=1$ for $s, t \in S$ with $m_{s, t}< \infty$. We say that $(W, S)$ is a {\it Coxeter group}. Sometimes we just call $W$ itself a Coxeter group.

Let $\Aut(W, S)$ be the group of automorphisms of the group $W$ that preserve $S$. Let $\Om$ be a group with a group homomorphism to $\Aut(W, S)$. Set $\tW=W\rtimes\Om$. Then an element in $\tW$ is of the form $w \d$ for some $w \in W$ and $\d \in \Om$. We have that $(w \d) (w' \d')=w \d(w') \d \d' \in \tW$ with $\d, \d' \in \Om$.

For $w \in W$ and $\d \in \Om$, we set $\ell(w \d)=\ell(w)$, where $\ell(w)$ is the length of $w$ in the Coxeter group $(W,S)$. Thus $\Om$ consists of length $0$ elements in $\tW$. We sometimes call the elements in $\Om$ {\it basic elements} in $\tW$.

We are mainly interested in the $W$-conjugacy classes in $\tW$. By \cite[Remark 2.1]{GKP}, for any $\d \in \Om$, the map $W \to \tW$, $w \mapsto w \d$ gives a bijection between the $\d$-conjugacy classes in $W$ and the $W$-conjugacy classes in $\tW$ that are contained in $W \d$.



\subsection{}\label{to} For $w, w' \in \tW$ and $s \in S$, we write $w \xrightarrow{s} w'$ if $w'=s w s$ and $\ell(w') \le \ell(w)$.  We write $w \to w'$ if there is a sequence $w=w_0, w_1, \cdots, w_n=w'$ of elements in $\tW$ such that for any $k$, $w_{k-1} \xrightarrow{s} w_k$ for some $s \in S$.

We write $w \approx w'$ if $w \to w'$ and $w' \to w$. It is easy to see that $w \approx w'$ if $w \to w'$ and $\ell(w)=\ell(w')$.

We call $\tw, \tw' \in \tW$ {\it elementarily strongly conjugate} if $\ell(\tw)=\ell(\tw')$ and there exists $x \in W$ such that $\tw'=x \tw x \i$ and $\ell(x \tw)=\ell(x)+\ell(\tw)$ or $\ell(\tw x \i)=\ell(x)+\ell(\tw)$. We call $\tw, \tw'$ {\it strongly conjugate} if there is a sequence $\tw=\tw_0, \tw_1, \cdots, \tw_n=\tw'$ such that for each $i$, $\tw_{i-1}$ is elementarily strongly conjugate to $\tw_i$. We write $\tw \sim \tw'$ if $\tw$ and $\tw'$ are strongly conjugate. We write $\tw \tilde \sim \tw'$ if $\tw \sim \d \tw'\d^{-1}$ for some $\d\in\Om$.

The following result is proved in \cite{GP}, \cite{GKP} and \cite{He1} via a case-by-case analysis with the aid of computer for exceptional types. A case-free proof which does not rely on computer calculation was recently obtained in \cite{HN}.

\begin{thm} Assume that $W_0$ is a finite Coxeter group. Let $\co$ be a conjugacy class in $\tW$ and $\co_{\min}$ be the set of minimal length elements in $\co$. Then

(1) For each $w \in \co$, there exists $w' \in \co_{\min}$ such that
$w \rightarrow w'$.

(2) Let $w, w' \in \co_{\min}$, then $w \sim w'$.
\end{thm}

\

The main purpose of this paper is to extend the above theorem to affine Weyl groups and to discuss its application to affine Hecke algebra. To do this, we first recall some basic facts on affine Weyl groups and Bruhat-Tits building.

\subsection{}\label{setup}
Let $\Phi$ be a reduced root system and $W_0$ the corresponding finite Weyl group. Then $(W_0, S_0)$ is a Coxeter group, where $S_0$ is the set of simple reflections in $W$.

Let $Q$ be the coroot lattice spanned by $\Phi^\vee$ and $$W=Q \rtimes W_0=\{t^\chi w; \chi \in Q, w \in W_0\}$$ be the affine Weyl group. The multiplication is given by the formula $(t^\chi w) (t^{\chi'} w')=t^{\chi+w \chi'} w w'$. Moreover, $(W, S)$ is a Coxeter group, where $S \supset S_0$ is the set of simple reflections in $W$.

The length function on $W$ is given by the following formula (see \cite{IM65}) $$\ell(t^\chi w)=\sum_{\a, w \i(\a) \in \Phi^+} |\<\chi, \a\>|+\sum_{\a\in \Phi^+, w \i(\a)\in \Phi^-} |\<\chi, \a\>-1|.$$

\subsection{} Let $V=Q \otimes_\ZZ \RR$. Then we have a natural action of $\tW$ on $V$. For $x, y \in V$, define $(x, y)=\sum_{\a \in \Phi} \<x, \a\> \<y, \a\>$. Then by \cite[Ch. VI, $\S$1, no.1,  Prop 3]{Bour}, $(\, , )$ is a positive-definite symmetric bilinear form on $V$ invariant under $W$. We define the norm $|| \cdot ||: V \to \RR$ by $||x||=\sqrt{(x, x)}$ for $x \in V$.

For $\a \in\Phi$ and $k\in\ZZ$, define $H_{\a,k}=\{x\in V; \<x,\a\>=k\}$. Let $\fH=\{H_{\a,k}; \ \a\in\Phi, k\in\ZZ\}$. For each hyperplane $H\in\fH$, let $s_H\in W$ be the reflection with respect to $H$. Connected components of $V-\cup_{H\in\fH}H$ are called {\it alcoves}. We denote by $\bar A$ the closure of an alcove $A$. We denote by $\D$ the fundamental alcove, i.e. the alcove in the dominant chamber such that $0 \in \bar \D$.

Let $H\in\fH$. If the set of inner points $H_A=(H\cap\bar{A})^\circ \subset H\cap\bar{A}$ spans $H$, then we call $H$ a {\it wall} of $A$ and $H_A$ a {\it face} of $A$.

Let $p,q\in V$ be two points, denote by $L(p,q)\subset V$ the affine subspace spanned by $p$ and $q$.

Let $K\subset V$ be a convex subset. We call a point $x \in K$ {\it regular} if for any $H \in \fH$, $x \in H$ implies that $K \subset H$. It is clear that all the regular points of $K$ form an open dense subset of $K$.

\subsection{} The action of $\tW$ on $V$ sends hyperplanes in $\fH$ to hyperplanes in $\fH$ and thus induces an action on the set of alcoves. It is known that the affine Weyl group $W$ acts simply transitively on the set of alcoves. For any alcove $A$, we denote by $x_A$ the unique element in $W$ such that $x_A \D=A$.

For any $\tw \in \tW$ and alcove $A$, set $\tw_A=x_A \i \tw x_A$. Then any element in the $W$-conjugacy class of $\tw$ is of the form $\tw_A$ for some alcove $A$.

For any two alcoves $A, A'$, let $\fH(A,A')$ denote the set of hyperplanes in $\fH$ separating them. Then $H \in \fH(\D, \tw \D)$ if and only if $\tw \i \a$ is a negative affine root, where $\a$ is the positive affine root corresponding to $H$. In this case, $\ell(s_H \tw)<\ell(\tw)$. We also have that $\ell(\tw)=\sharp\fH(\D, \tw\D)$.

\section{Minimal length elements in affine Weyl groups}

Unless otherwise stated, we write $W_0$ for finite Weyl group and $W$ for affine Weyl group in the rest of this paper.

\subsection{}\label{alcove}
Similar to \cite{HN}, we may view conjugation by simple reflections in the following way.

Let $A, A'$ be two alcoves with a common face $H_A=H_{A'}$, here $H \in \fH$. Let $s_H$ be the reflection along $H$ and set $s=x_A \i s_H x_A$. Then $s \in S$. Now $$\tw_{A'}=(s_H x_A) \i \tw (s_H x_A)=s x_A \i \tw x_A s=s \tw_A s$$ is obtained from $\tw_A$ by conjugating the simple reflection $s$. We may check if $\ell(\tw_{A'})>\ell(\tw_A)$ by the following criterion.

\begin{lem}\label{Dr}
We keep the notations as above. Define $f_{\tw}:V\rightarrow\RR$ by $v\mapsto||\tw(v)-v||^2$. Let $h$ be a regular point in $H_A$ and $v\in V$ such that $(v, h-h')=0$ for all $h' \in H_A$ and $h-\e v\in A$ for sufficient small $\e>0$. Set $\Dr_vf_{\tw}(h)=\lim_{t\rightarrow0}\frac{f_{\tw}(h+tv)-f_{\tw}(h)}{t}$. If $\ell(\tw_{A'})=\ell(s\tw_As)=\ell(\tw_A)+2$, then $\Dr_vf_{\tw}(h)>0$.
\end{lem}
\begin{proof}
It is easy to see that $\fH(A',\tw A')-\fH(A,\tw A)\subset\{H,\tw H\}$. By our assumption, $\sharp\fH(A',\tw A')=\sharp\fH(A,\tw A)+2$. Hence $$\fH(A',\tw A')=\fH(A,\tw A)\sqcup\{H,\tw H\}$$ and $H \neq\tw H$. In particular, $h\neq\tw(h)$ since $h$ and $\tw(h)$ are regular points of $H$ and $\tw H$ respectively.

Moreover, $L(h,\tw(h))-\{h, \tw(h)\}$ consists of three connected components: $L_{-}=\{h+t(\tw(h)-h);\ t<0\}$,  $L_{0}=\{h+t(\tw(h)-h);\ 0<t<1\}$ and $L_{+}=\{h+t(\tw(h)-h);\ t>1\}$. Note that $\fH(A,\tw A)\cap\{H,\tw H\}=\emptyset$, $A\cap L_0$ and $\tw A\cap L_0$ are nonempty. Since $v$ is perpendicular to $H$ and $h-v, \tw(h)$ lie in the same connected component of $V-H$, we have $(v,h-\tw(h))>0$. Similarly $(\tw(v)-\tw(0),\tw(h)-h)>0$.
Now $\tw(h+tv)=\tw(h)+t(\tw(v)-\tw(0))$ and $\tw(h+tv)-t v-h=(\tw(h)-h)+t(\tw(v)-\tw(0)-v)$. Hence
\begin{align*}
\Dr_v f_{\tw}(h) &= 2 (\tw(h)-h, \tw(v)-\tw(0)-v) \\ &=2(\tw(h)-h, \tw(v)-\tw(0))+2(h-\tw(h),v)>0.
\end{align*}
\end{proof}

\subsection{} Let $\gr f_{\tw}$ denote the gradient of the function $f_{\tw}$ on $V$, that is, for any other vector field $X$ on $V$, we have $Xf_{\tw}=(X,\gr f_{\tw})$. Here we naturally identify $V$ with the tangent space of any point in $V$.

We'll describe where the gradient vanishes. To do this, we introduce an affine subspace $V_{\tw}$.

Notice that $\Om$ is a finite subgroup of $\tW$. For any $\tw \in \tW$, there exists $n \in \NN$ such that $\tw^n \in W$. Hence there exists $m \in \NN$ such that $\tw^{m n}=t^\l$ for some $\l \in Q$. Set $\nu_{\tw}=\l/m n \in V$ and call it the {\it Newton point of} $\tw$. Then it is easy to see that $\nu_{\tw}$ doesn't depend on the choice of $m$ and $n$. We set \[V_{\tw}=\{v \in V; \tw(v)=v+\nu_{\tw}\}.\]

\begin{lem}\label{plusv}
Let $\tw\in\tW$. Then $V_{\tw}\subset V$ is a nonempty affine subspace such that $V_{\tw}=\tw V_{\tw}=V_{\tw}+\nu_{\tw}$.
\end{lem}
\begin{proof}
Since $\tw$ is an affine transformation, for any $p \neq q \in V_{\tw}$, the affine line $L(p, q)$ is also contained in $V_{\tw}$. Thus $V_{\tw}$ is an affine subspace of $V$.

Now we prove that $V_{\tw}$ is nonempty. Assume $\tw^n=t^{n\nu_{\tw}}$ for some $n>0$. Let $q \in V$. Set $p=\frac{1}{n}\sum_{i=0}^{n-1}\tw(q)$. Then $\tw(p)-p=\frac{1}{n}(\tw^n(p)-p)=\nu_{\tw}$. In particular, $V_{\tw} \neq \emptyset$.

For any $x\in V_{\tw}$, $||\tw^k(x)-\tw^{k-1}(x)||=||\tw(x)-x||=||\nu_{\tw}||$ and $\tw^n(x)-x=n\nu_{\tw}$. Hence $\tw^k(x)=\tw^{k-1}(x)+\nu_{\tw}$ for all $k\in\ZZ$. In particular, $\tw(x)=x+\nu_{\tw}\in V_{\tw}$.
\end{proof}

\begin{lem}
Let $v\in V$. Then $\gr f_{\tw}(v)=0$ if and only if $v\in V_{\tw}$.
\end{lem}
\begin{proof}
Let $p\in V_{\tw}$. Set $T_p: V\rightarrow V$ by $v\mapsto v+p$. Define $\tw_p=T_p^{-1}\circ\tw\circ T_p$ and $V_{\tw_p}=T_p \i V_{\tw}=\{v\in V; \ \tw_p(v)=v+\nu_{\tw}\}$. Then $V_{\tw_p}\subset V$ is a linear subspace. Let $V_{\tw_p}^\bot=\{v\in V; (v,V_{\tw_p})=0\}$ be its orthogonal complement.

Then $\tw_{p}(x+y)=\tw_p(x)+\tw_p(y)-\tw_p(0)=x+\tw_p(y)$ for any $x \in V_{\tw_p}$ and $y \in V_{\tw_p}^\bot$. Since $\tw_p$ is an isometry on $V$, we have that \begin{align*} ||x||^2+||y||^2 &=||x+y||^2=||\tw_p(x+y)-\tw_p(0)||^2 \\ &=||x||^2+||\tw_p(y)-\tw_p(0)||^2+2(x, \tw_p(y)-\tw_p(0)).\end{align*}

In particular, $(x, \tw_p(y)-\tw_p(0))=0$ for all $x \in V_{\tw_p}$ and $y \in V_{\tw_p}^\bot$. Hence $\tw_p(y)-\tw_p(0) \in V_{\tw_p}^\bot$ for all $y \in V_{\tw_p}^\bot$. Let $M: V_{\tw_p}^\bot \to V_{\tw_p}^\bot$ be the linear transformation defined by $y \mapsto \tw_p(y)-\tw_p(0)-y$. By definition, $\ker M \subset V_{\tw_p} \cap V_{\tw_p}^\bot=\{0\}$ for $\tw_p(0)=\nu_{\tw}$. Hence $M$ is invertible.

For $x \in V_{\tw_p}$ and $y \in V_{\tw_p}^\bot$, \begin{align*}f_{\tw_p}(x+y)&=||\tw_p(x+y)-(x+y)||^2=||M(y)||^2+||\tw_p(0)||^2\\ &=(y, {}^t M M(y))+ ||\nu_{\tw}||^2,\end{align*} where ${}^t M$ is the transpose of $M$ with respect to the inner product $(, )$ on $V$. Thus $\gr f_{\tw_p}(x+y)=2{}^tMM(y)\in V$. Hence it vanishes exactly on $V_{\tw_p}$.

Notice that $f_{\tw_p}=f_{\tw}\circ T_p$ and $T_p$ is an isometry. We have that $\gr f_{\tw}(v)=\gr f_{\tw_p}(v-p)$ for any $v\in V$. Hence $\gr f_{\tw}$ vanishes exactly on $V_{\tw}$.
\end{proof}

\begin{cor}
Let $C_{\tw}:V\times \RR\rightarrow V$ denote the integral curve of the vector field $\gr f_{\tw}$ with $C_{\tw}(v,0)=v$ for all $v \in V$. Define $\Lim:V\rightarrow V_{\tw}$ by $v\mapsto\lim_{t\rightarrow-\infty}C_{\tw}(v,t)\in V_{\tw}$. Then $\Lim:V\rightarrow V_{\tw}$ is a trivial vector bundle over $V_{\tw}$.
\end{cor}
\begin{proof}
We keep notations in the proof of the above lemma. The integral curve of $\gr f_{\tw_p}$ can be written explicitly as
$C_{\tw_p}(x+y,t)=x+\Exp(2t{}^tMM)(y)$ for any $x \in V_{\tw_p}$ and $y \in V_{\tw_p}^\bot$. Hence the integral curve $C_{\tw}(t,v)$ of $\gr f_{\tw}$ is given by $C_{\tw}(x+y,t)=x+\Exp(2t{}^tMM)(y)$ for $x \in V_{\tw}$ and $y \in V_{\tw_p}^\bot$.
Since ${}^tMM$ is self adjoint with positive eigenvalues, $\lim_{t\rightarrow-\infty}\Exp(2t{}^tMM)=0$. Hence $\Lim(x+y)=x$ for any $x\in V_{\tw}$ and $y\in V_{\tw_p}^\bot$. Thus $\Lim$ is a trivial vector bundle over $V_{\tw}$.
\end{proof}

\begin{prop}\label{red1}
Let $\tw\in\tW$ and $A$ be an alcove. Then there exists an alcove $A'$ such that $\bar A'$ contains a regular point of $V_{\tw}$ and $\tw_A\rightarrow \tw_{A'}$.
\end{prop}
\begin{proof}
Let $V_{\tw}^{\ge 1}\subset V_{\tw}$ be the complement of the set of regular points of $V_{\tw}$. By the above corollary, $\Lim^{-1}(V_{\tw}^{\ge 1})$ is a countable union of submanifolds of $V$ of codimension at least $1$. Let $V^{\ge 2}$ be the complement of all alcoves and faces in $V$, that is, the skeleton of $V$ of codimension $\ge 2$. Then $C_{\tw}(V^{\ge 2},\RR)\subset V$ is a countable union of images, under smooth maps, of manifolds of dimension at most $\dim V-1$. Set $$D_{\tw}=\{v \in V; v \notin C_{\tw}(V^{\ge 2},\RR) \cup \Lim^{-1}(V_{\tw}^{\ge 1})\}.$$ Then $D_{\tw}$ is dense in $V$ in the sense of Lebesgue measure.

Choose $y \in A \cap D_{\tw}$ and set $x=\Lim(y)\in V_{\tw}$. Then $x$ is a regular point of $V_{\tw}$. There exists  $T>0$ such that $x \in \bar B$ if $C_{\tw}(y, -T) \in \bar B$ for any alcove $B$.

Now we define $A_i, H_i, h_i, t_i$ as follows.

Set $A_0=A$. Suppose $A_i$ is defined, then we set $t_i=\sup\{t<T; C_{\tw}(y, -t) \in \bar{A}_i\}$. If $t_i=T_i$, the procedure stops. Otherwise, set $h_i=C_{\tw}(y, -t_i)$. By the definition of $D_{\tw}$, $h_i$ is contained in a unique face of $A_i$, which we denote by $H_i$. Let $A_{i+1} \neq A_i$ be the unique alcove such that $H_i$ is a common face of $A_i$ and $A_{i+1}$. Then $C_{\tw}(y, -t_i-\e) \in A_{i+1}$ for sufficiently small $\e>0$.

Since $\{C_{\tw}(y, -t); 0 \le t \le T\}$ is compact, it intersects with only finitely many alcove closures. Note also that alcoves appear in the above list are distinct with each other.  Thus the above procedure stops after finitely many steps. We obtain a finite sequence of alcoves $A=A_0, A_1, \cdots, A_r=A'$ in this way. Then $x \in \bar A'$ for $C_{\tw}(y, -T) \in \bar A'$.

Let $v_i \in V$ such that $(v_i, h_i-h)=0$ for $h \in H_i$ and $h_i-\e v_i \in A_i$ for sufficiently small $\e>0$. Since $C_{\tw}(y,-t_i-\e')\in A_{i+1}$ for sufficiently small $\e'>0$, $\Dr_{v_i} f_{\tw}(h_i)=(v_i,(\gr f_{\tw})(h_i))\le0$. Hence by Lemma \ref{Dr}, $\ell(\tw_{A_{i+1}})\le \ell(\tw_{A_i})$ and $\tw_{A_i}\rightarrow\tw_{A_{i+1}}$.

Therefore $\tw_A \to \tw_{A'}$ and $\bar A'$ contains a regular point $x$ of $V_{\tw}$.
\end{proof}

\subsection{}
As a consequence, there exists a minimal length element in the conjugacy class of $\tw$ which is of the form $\tw_A$ for some alcove $A$ with $V_{\tw} \cap \bar A \neq \emptyset$. However, not every minimal length element is of this form. Now we give an example.

Let $W$ be the affine group of type $\tilde A_2$ with a set of simple reflections $\{s_1=s_{\a_1},s_2=s_{\a_2},s_0=t^{\a_1+\a_2}s_1s_2s_1\}$, where $\a_1,\a_2$ are the two distinct simple positive roots. The corresponding fundamental coweight is denoted by $\o_1, \o_2$ respectively. Let $\d\in\Aut W$ such that $\d:s_1\mapsto s_2, s_2\mapsto s_1, s_0\mapsto s_0$. Let $\tw=t^{2\a_1+2\a_2}\d$ and $A=t^{\a_1}s_1s_2\D$. Then $\ell(\tw)=\ell(\tw_A)=8$ and $\tw_A$ is of minimal length in $W \cdot \tw$. Note that $V_{\tw}=\{v\in V; (v,\a_1-\a_2)=0\}$. The vertices (extremal points) of $\bar A$ are $\o_1$, $\o_1-\o_2$ and $2\o_1-\o_2$ which all lie in the same connected component of $V-V_{\tw}$. Hence $V_{\tw} \cap \bar A=\emptyset$.

\subsection{} Now we recall the ``partial conjugation action'' introduced in \cite{He1}.

For $J \subset S$, we denote by $W_J$ the standard parabolic subgroup of $W$ generated by $J$ and by ${}^J \tW$ the set of minimal coset representatives in $W_J \backslash \tW$.

For $\tw \in {}^J \tW$, set $$I(J, \tw)=\max\{K \subset J; \tw(K)=K\}.$$ 

The following result is proved in \cite[Section 2 \& 3]{He1}. See also \cite[Theorem 2.1]{He2}.

\begin{thm}\label{par}
Let $J \subset S$ such that $W_J$ is finite. We consider the (partial) conjugation action of $W_J$ on $\tW$. Let $\co$ be an orbit. Then

(1) There exists $\tw \in {}^J \tW$, such that for any $\tw' \in \co$, there exists $x \in W_{I(J, \tw)}$ such that $\tw' \to x \tw$.

(2) If $\tw', \tw'' \in \co_{\min}$, then $\tw' \sim \tw''$.
\end{thm}

\subsection{}
We'll show that $\tw_{A'}$ appeared in Prop \ref{red1} is of the form $x y$ for some $y \in {}^J \tW$ and $x \in W_{I(J, \tw)}$. To do this, we introduce some more notations.

Let $K\subset V$ be a convex subset. Let $\fH_K=\{H\in\fH; K\subset H\}$ and $W_K\subset W$ be the subgroup generated by $s_H$ ($H\in\fH_K$). For any two alcoves $A$ and $A'$, define $\fH_K(A,A')=\fH(A,A')\cap\fH_K$.

Let $A$ be an alcove. We set $W_{K,A}=x_A^{-1}W_K x_A$. If $\bar A$ contains a regular point $v$ of $K$, then $W_{K,A}=W_{I(K, A)}$ is the subgroup of $W$ generated by simple reflections $I(K, A)=\{s_H\in S; v \in x H\}$.

\begin{lem}\label{f1}
Let $\tw \in \tW$ and $K \subset V_{\tw}$ be an affine subspace with $\tw(K)=K$. Let $A$ be an alcove such that $A$ and $\tw A$ are in the same connected component of $V-\cup_{H \in \fH_K}H$. Assume furthermore that $\bar A$ contains an element $v \in K$ such that for each $H \in \fH$, $v, \tw(v) \in H$ implies that $K \subset H$. Then
$$\ell(\tw_A)=\sharp \fH(A, \tw A)=\<\bar \nu_{\tw}, 2\rho\>.$$ Here $\rho$ is the half sum of the positive roots in $\Phi$ and $\bar{\nu}_{\tw}$ is the unique dominant element in the $W_0$-orbit of $\nu_{\tw}$.
\end{lem}

\begin{proof}
By our assumption, $\fH(A, \tw A) \subset \fH-\fH_K$. Moreover, for any $H \in \fH(A, \tw A)$, the intersection of $H$ with the closed interval $[v, \tw(v)]$ is nonempty.

If $\nu_{\tw}=0$, then $\tw(v)=v$. For any $H \in \fH(A, \tw A)$, we have $v \in H$, hence $H \in \fH_K$. That is a contradiction. Hence $\fH(A, \tw A)=\emptyset$ and $\ell(\tw_A)=\<\bar \nu_{\tw}, 2\rho\>=0$.

Now we assume $\nu_{\tw} \neq 0$.  Set $v_i=\tw^i(v)=v+i \nu_{\tw} \in K$ for $i \in \ZZ$. Then all the $v_i$ span an affine line $L$. We prove that

(a) If $i<j$, then $\fH(\tw^{i-1} A, \tw^i A) \cap \fH(\tw^{j-1} A, \tw^j A)=\emptyset$.

Let $H \in \fH(\tw^{i-1} A, \tw^i) \cap \fH(\tw^{j-1} A, \tw^j A)$. Then $H \cap L=H \cap [v_{i-1}, v_i]  \cap [v_{j-1}, v_j] \neq \emptyset$. Thus $i=j-1$ and $v_i \in H$. Hence $H \in \fH(\tw^{i-1} A, \tw^j A)$. Therefore $\tw^i A$ and $\tw^j A$ are in the same connected component of $V-H$, that is, $H \notin \fH(\tw^{j-1} A, \tw^j A)$.

(a) is proved.

Now we prove that

(b) For $i<j$,  $\fH(\tw^{i} A, \tw^j)=\bigcup_{k=i+1}^j\fH(\tw^{k-1} A, \tw^k A)$.

If $H \notin \bigcup_{k=i+1}^j\fH(\tw^{k-1} A, \tw^k A)$, then $\tw^i A, \tw^{i+1} A, \cdots, \tw^j A$ are all in the same connected component of $V-H$. Thus $H \notin \fH(\tw^{i} A, \tw^j)$.

Let $H \in \fH(\tw^{r-1} A, \tw^r A)$ for some $i < r \leq j$. Then $H \cap L=H \cap [v_{r-1}, v_r]=\{e\}$. If $e \notin \{v_i, v_j\}$, then $H \in \fH(\tw^i A, \tw^j A)$. If $e=v_j$, then $H \in \fH(\tw^{j-1} A, \tw^j A)$ and $v_{j-1}$ and $v_i$ are in the same connected component of $V-H$. Hence $\tw^i A$ and $\tw^{j-1}A$ are in the same connected component of $V-H$, while $\tw^{j-1} A$ and $\tw^j A$ are in different connected components of $V-H$. Hence $H \in \fH(\tw^i A, \tw^j A)$. If $e=v_i$, by a similar argument we have that $H \in \fH(\tw^i A, \tw^j A)$.

Let $n \in \ZZ$ such that $\tw^n=t^{n \nu_{\tw}}$. Then by (a) and (b), we have that
$$n \sharp \fH(A, \tw A)= \sum_{i=0}^{n-1} \sharp \fH(\tw^i(A), \tw^{i+1}(A))=\sharp \fH(A, \tw^n  A)=\sharp \fH(A, t^{n \nu_{\tw}}A).$$
Hence $\ell(\tw_A)=\<\nu_{\tw}, 2\rho\>$.
\end{proof}

\begin{prop}\label{ell}
Let $\tw\in\tW$ and $K\subset V_{\tw}$ be an affine subspace with $\tw(K)=K$ . Let $A$ be an alcove such that $\bar A$ contains a regular point $v$ of $K$. Then $\tw_A=u \tw_{K, A}$ for some $u \in W_{I(K, A)}$ and $\tw_{K, A} \in {}^{I(K, A)} \tW^{I(K, A)}$ with $\ell(u)=\sharp\fH_K(A, \tw A)$, $\tw_{K, A} I(K, A)=I(K, A)$ and $\ell(\tw_{K, A})=\<\overline{\nu}_{\tw },2\rho\>$.
\end{prop}
\begin{proof}
We may assume that $A$ is the fundamental alcove $\D$ by replacing $\tw$ by $\tw_A$. We simply write $I$ for $I(K, \D)$.

If $\nu_{\tw}=0$, then $\tw(v)=v \in \bar \D \cap \tw \bar \D$. If $H \in \fH$ is a hyperplane separating $\D$ and $\tw \D$, then $v \in H$. Since $v$ is regular in $K$, $K \subset H$. Thus $H \in \fH_K$ and $s_H \in W_I$. Thus there exists $u \in W_I$ such that $u \i \tw \D=\D$. In this case, $\tw_{K, \D}=u \i \tw$ is a length-zero element in $\tW$ and $\tw=u \tw_{K, \D}$.

Now we assume that $\nu_{\tw} \neq 0$. We have that $\tw=u' \tw' u''$ for some $u', u'' \in W_I$ and $\tw' \in {}^I \tW^I$.
Since $\tw(K)=K$, then $\tw(\fH_K)=\fH_K$ and $\tw W_I \tw \i=W_I$. Hence $\tw' W_I (\tw') \i=W_I$ and $\tw'(I)=I$.

Let $C$ be the connected component of $V-\cup_{H\in\fH_K}H$ that contains $\D$. We have that $\tw' (\D) \subset C$. Otherwise, there exists $H \in \fH_K$ separating $\D$ and $\tw'(\D)$. Hence $\ell(s_H \tw')<\ell(\tw')$.  This contradicts our assumption that $\tw' \in {}^I \tW$. Hence $\ell(u)=\sharp\fH_K(\D, \tw \D)$ and $\fH(\D, \tw'(\D)) \cap \fH_K=\emptyset$.

Since $W_I$ is a finite group and the conjugation by $\tw$ is a group automorphism on $W_K$, there exists $n>0$ such that $$(\tw')^n \tw^{-n}=u \i (\tw u \i \tw \i) \cdots (\tw^{n-1} u \i \tw^{-n+1})=1.$$ Hence $(\tw')^n=\tw^n$ and there exists $m>0$ such that $m n \nu_{\tw} \in Q$ and $(\tw')^{mn}=\tw^{mn}=t^{m n \nu_{\tw}}$.

Note that $v$ and $\tw(v)=\tw'(v)=v+\nu_{\tw}$ are regular points in $K$. Apply Lemma \ref{f1}, we have $\ell(\tw')=\<\bar \nu_{\tw'}, 2\rho\>=\<\bar \nu_{\tw}, 2\rho\>$.
\end{proof}

\begin{cor}\label{finiteW}
Let $\tw\in\tW$ be of minimal length in its conjugacy class. Then $\tw$ is of finite order if and only if $\tw\in W_J\rtimes\<\d\>$ for some proper subset $J$ of $S$ and $\d\in\Om$ with $\d(J)=J$ such that the corresponding parabolic subgroup $W_J$ is finite.
\end{cor}
\begin{proof}
The ``if'' part is clear.

Now assume that $\tw$ is of finite order. Let $K=V_{\tw}$. By Proposition \ref{red1} and \ref{ell}, there exists an alcove $A$ such that $\tw \approx \tw_A$ and $\tw_A=u \tw_{K, A}$ for some $u \in W_{I(K, A)}$ and $\tw_{K, A} \in {}^{I(K, A)} W$ with $\tw_{K, A} I(K, A)=I(K, A)$ and $\ell(\tw_{K, A})=\<\overline{\nu}_{\tw },2\rho\>$. Since $\tw$ is of finite order, $\nu_{\tw}=0$ and $\ell(\tw_{K, A})=0$. So $\tw_{K, A} \in \Om$.

By definition, $I(K, A)$ is a subset of $S$ such that $W_{I(K, A)}$ is finite. We have that $\tw_A \in W_{I(K, A)} \tw_{K, A}$ and $\tw \approx \tw_A$. Hence $\tw \in W_{I(K, A)} \tw_{K, A}$.
\end{proof}

\

Now we may prove the main result of this section, generalizing $\S$0.1 (1) and (2) to affine Weyl groups.

\begin{thm}\label{main} Let $\co$ be a $W$-conjugacy class in $\tW$ and $\co_{\min}$ be the set of minimal length elements in $\co$. Then

(1) For each element $\tw' \in \co$, there exists $\tw'' \in \co_{\min}$ such that $\tw' \rightarrow \tw''$.

(2) Let $\tw', \tw'' \in \co_{\min}$, then $\tw' \sim \tw''$.
\end{thm}

\begin{proof}
(1) We fix an element $\tw$ of $\co$. Set $K=V_{\tw}$. Then any element in $\co$ is of the form $\tw_{A'}$ for some alcove $A'$. By Proposition \ref{red1}, $\tw_{A'} \to \tw_{A}$ for some alcove $A$ such that $\bar A$ contains a regular point of $K$.

If $C$ is a connected component of $V-\cup_{H \in \fH_K} H$, then $\tw(C)$ is also a connected component of $V-\cup_{H \in \fH_K} H$. We denote by $\ell(C)$ the number of hyperplanes in $\fH_K$ that separate $C$ and $\tw(C)$. By definition, if $C$ is the connected component that contains $A$, then $\ell(C)=\sharp \fH_K(A, \tw A)$. Now by Proposition \ref{ell}, $\ell(\tw_A)=\ell(C)+\<\bar \nu_{\tw}, 2 \rho\>$.

Let $C_0$ be a connected component of $V-\cup_{H \in \fH_K} H$ such that $\ell(C_0)$ is minimal among all the connected components of $V-\cup_{H \in \fH_K} H$. Then $\ell(\tw_{A'}) \ge \ell(\tw_A) \ge \ell(C_0)+\<\bar \nu_{\tw}, 2 \rho\>$. In particular, let $A_0$ be the alcove in $C_0$ such that $A_0=u A$ for some $u \in W_K$, then $\bar A_0$ contains a regular point of $K$ and $\tw_{A_0} \in \co_{\min}$.

We have that $\tw_{A_0}=u' \tw_A (u') \i$ for some $u' \in W_{K, A}$ and $\tw_{A_0}$ is a minimal length element in $\co'=\{x \tw_A x \i; x \in W_{K, A}\} \subset \co$. Hence by Theorem \ref{par}, there exists $\tw'' \in \co'_{\min}$ such that $\tw_{A'} \to \tw_A \to \tw'' \sim \tw_{A_0}$. Since $\tw_{A_0} \in \co_{\min}$, $\tw'' \in \co_{\min}$. Part (1) is proved.

(2) Let $\tw' \in \co_{\min}$. We have showed that there exists an alcove $A_0' \subset C_0$ such that $\bar A'_0$ contains a regular point of $K$ and $\tw' \sim \tw_{A'_0}$. Now it suffices to prove that $\tw_{A_0} \sim \tw_{A'_0}$.

Let $\ca_{C_0}$ be the set of all alcoves in $C_0$ whose closures contain regular points of $K$. Then $\cup_{A \in \ca_{C_0}} \bar A \supset K$. Hence there exists a finite sequence of alcoves $A=A_0,\cdots,A_r=A_0' \in \ca_{C_0}$ such that $\bar A_i \cap \bar A_{i+1} \cap K \neq \emptyset$ for all $0 \le i<r$. We set $K_i=\bar A_i \cap \bar A_{i+1} \cap K$. Then there exists $u_i \in W_{K_i}$ such that $A_{i+1}=u_i A_i$. Hence $\tw_{A_{i+1}}=u'_i \tw_{A_i} (u'_i) \i$ for some $u'_i \in W_{K_i, A_i}$. Notice that $\tw_{A_{i+1}}$ and $\tw_{A_i}$ are minimal length elements in $\{x \tw_{A_i} x \i; x \in W_{K_i, A_i}\}$. Thus by Theorem \ref{par}, $\tw_{A_{i+1}} \sim \tw_{A_i}$. Therefore $\tw_{A_0} \sim \tw_{A'_0}$.
\end{proof}

\

As a consequence, we have a similar result for any conjugacy class of $\tW$, which is a union of $W$-conjugacy classes.

\begin{cor}\label{maincor}
Let $\co$ be a conjugacy class of $\tW$ and $\co_{\min}$ be the set of minimal length elements in $\co$. Then

(1) For each element $\tw' \in \co$, there exists $\tw'' \in \co_{\min}$ such that $\tw' \rightarrow \tw''$.

(2) Let $\tw', \tw'' \in \co_{\min}$, then $\tw' \tilde \sim \tw''$.
\end{cor}

\section{Straight conjugacy class}

\subsection{}  Following \cite{Kr}, we call an element $\tw \in \tW$ a {\it straight element} if for any $m \in \NN$, $\ell(\tw \d(\tw) \cdots \d^{m-1}(\tw))=m \ell(\tw)$. We call a conjugacy class {\it straight} if it contains some straight element. It is easy to see that $\tw$ is straight if and only if $\ell(\tw)=\<\bar \nu_{\tw}, 2 \rho\>$ (see \cite{He2}).

By definition, any basic element of $\tW$ is straight. Also if $\l$ is dominant, then $t^\l$ is also straight. In Proposition \ref{coxeter}, we'll give some nontrivial example of straight elements.

\subsection{} We follow \cite[7.3]{Sp}. Let $\d \in \Om$. For each $\delta$-orbit in $S$, we pick a simple reflection. Let $g$ be the product of these simple reflections (in any order) and put $c=(g, \delta) \in W \rtimes \<\delta\>$. We call $c$  a {\it twisted Coxeter element} of $\tW$. The following result will be used in \cite{HW} in the study of basic locus of Shimura varieties.

\begin{prop}\label{coxeter}
Let $c$ be a twisted Coxeter element of $\tW$. Then $c$ is a straight element.
\end{prop}

\begin{rmk}
The case where $\d=1$ (for any Coxeter group of infinite order) was first obtained by Speyer in \cite{Spe}. Our method here is different from loc.cit.
\end{rmk}

\begin{proof}
Assume that $c \in W \d$ for $\d \in \Om$. By Proposition \ref{red1} and \ref{ell}, $c \rightarrow u x$ for some $I \subset S$ with $W_I$ finite, $u \in W_I$, and a straight element $x$ with $x(I)=I$. It is easy to see that $u x$ is also a twisted Coxeter element and $c \approx u x$. In particular, $x=w \d$ for some $w \in W_{S-I}$. For any $s \in I$, $w \d(s) w \i \in I$. Hence $\d(s) \in I$ and commutes with $w$. So $\d(I)=I$ and $\d(S-I)=S-I$. Since $u x$ is a twisted Coxeter element of $\tW$, $x=w \d$ is a twisted Coxeter element of $W_{S-I} \rtimes \<\d\>$. On the other hand, $w$ commutes with any element in $I$. Thus $I$ is a union of connected components of the Dynkin diagram of $S$. Hence $I=\emptyset$ since $W_I$ is finite. So $u=1$ and $c \approx x$ is also a straight element.
\end{proof}

\subsection{}\label{wgo} We'll give some algebraic and geometric criteria for straight conjugacy classes. In order to do this, we first make a short digression and discuss another description of $\tW$.

Let $G$ be a connected complex reductive algebraic group and $T \subset G$ be a maximal torus of $G$. Let $W_0$ be the finite Weyl group of $G$ and $S_0$ the set of simple roots. We denote by $Q$ (resp. $P$) the coroots lattice (resp. coweight lattice) of $T$ in $G$. Then $W_G=Q \rtimes W_0$ is an affine Weyl group in \ref{setup}. Set $\tW_G=P \rtimes W_0$. For the group $\Om'$ of diagram automorphisms of $S_0$ that induces an action on $G$, we set $\tW_{G, \Om'}=\tW_G \rtimes \Om'$. Then $\tW_{G, \Om'}=W_G \rtimes \Om$ for some $\Om \subset \Aut(W_G, S)$ with $\Om(S)=S$. It is easy to check that for any affine Weyl group $W$, $W \rtimes \Aut(W, S)=\tW_{G, \Om'}$, here $G$ is the corresponding semisimple group of adjoint type and $\Om'$ is the group of diagram automorphisms on $S_0$.

For any $J \subset S_0$, set $\Om'_J=\{\d \in \Om'; \d(J)=J\}$ and $$\tW_J=(P \rtimes W_J) \rtimes \Om'_J.$$ We call an element in $\tW_J$ {\it basic} if it is of length $0$ with respect to the length function on $\tW_J$.

In the rest of this section, we assume that $W=W_G$ and $\tW=\tW_{G, \Om'}$ unless otherwise stated.

\begin{prop}\label{cri}
Let $\co$ be a $W_G$-conjugacy class of $\tW$. Then the following conditions are equivalent:

(1) $\co$ is straight;

(2) For some (or equivalently, any) $\tw \in \co$, $V_{\tw} \nsubseteq H$ for any $H \in \fH$;

(3) $\co$ contains a basic element of $\tW_J$ for some $J \subset S_0$.

In this case, there exists a basic element $x$ in $\tW_{J_{\co}}$ and $y \in W_0^{J_\co}$ such that $\nu_x=\nu_\co$ and $y x y\i \in \co_{\min}$. Here $\nu_\co=\bar \nu_{\tw}$ for some (or equivalently, any) $\tw \in \co$ and $J_{\co}=\{i \in S_0; \<\nu_\co, \a_i\>=0\}$.
\end{prop}

\begin{proof}
(1) $\Leftrightarrow$ (2). By Proposition \ref{red1} and Proposition \ref{ell}, there is an alcove $A$ such that $\bar A$ contains a regular point of $V_{\tw}$ and $\tw_A \in \co_{\min}$. Moreover $\ell(\tw_A)=\<\bar \nu_{\tw}, 2\rho\>+\sharp\fH_{V_{\tw}}(A,\tw A)$.

If $\fH_{V_{\tw}}=\emptyset$, then $\fH_{V_{\tw}}(A, \tw A)=\emptyset$. Hence $\ell(\tw_A)=\<\bar \nu_{\tw}, 2\rho\>$ and $\tw$ is a straight element.

If $\co$ is straight, then $\sharp\fH_{V_{\tw}}(A, \tw A)=0$, that is, $\tw$ fixes the connected component $C$ of $V-\cup_{H \in \fH_{V_{\tw}}}H$ containing $A$. Choose $v \in C$ and set $y=\frac{1}{n}\sum_{k=0}^{n-1}\tw^k(v)$, where $n \in \NN$ with $\tw^n=t^{n \nu_{\tw}}$. Since $C$ is convex, we have $y \in C \cap V_{\tw}$, which forces $\fH_{V_{\tw}}$ to be empty.

(3) $\Rightarrow$ (2). Denote by $\Phi_J \subset \Phi$ the set of roots spanned by $\a_i$ for $i \in J$. Assume $\tw=t^\chi w \d' \in \co$ is a basic element in $\tW_J$. Then it is a straight element in $\tW_J$. By condition (2) for $\tW_J$, $V_{\tw} \nsubseteq H_{\a, k}$ for any $\a \in \Phi_J$ and $k \in \ZZ$.

Let $\mu \in V$ with $\<\mu, \a_i\>=\begin{cases} 0, & \text{ if } i \in J \\ 1, &\text{ if } i\in S_0-J\end{cases}$. Since $\d'(J)=J$, then $\d'(\mu)=\mu$. Hence $w \d'(\mu)=\mu$ and $\RR \mu + V_{\tw}=V_{\tw}$. Therefore $\<V_{\tw}, \a\>=\RR$ for any $\a \in \Phi-\Phi_J$. Thus $V_{\tw} \nsubseteq H_{\a, k} \in \fH$ with $\a \in \Phi-\Phi_J$ and $k \in \ZZ$.

(1) $\Rightarrow$ (3). By Proposition \ref{ell} and Condition (2)\, there exists $\tw \in \co_{\min}$ such that $\D$ contains a regular point $e$ of $V_{\tw}$. Let $y \in W_0^{J_\co}$ with $\nu_\co=y\i(\nu_{\tw})$. Set $x=y \i \tw y$. Then $\nu_x=\nu_\co$ is dominant.

Assume that $x=t^\chi w \d' \in \co$ with $\chi \in P$, $w \in W_0$ and $\d' \in \Om'$. Let $n \in \NN$ with $x^n=t^{n \nu_{\co}}$. Then $$t^{n \nu_{\co}+\chi} w \d'=t^{n \nu_{\co}} x=x t^{n \nu_{\co}}=t^{w \d'(\nu_{\co})+\chi} w \d'.$$ Thus $\nu_{\co}=w \d'(\nu_{\co})$ is the unique dominant element in $W_0 \cdot \d'(\nu_{\co})$. Hence $\d'(\nu_{\co})=\nu_{\co}$ and $w \nu_{\co}=\nu_{\co}$. Therefore $w \in W_{J_{\co}}$ and $\d'(J_\co)=J_\co$. Hence $x \in \tW_{J_\co}$.

Let $C$ be the connected component of $V-\cup_k \cup_{\a \in \Phi_J} H_{\a, k}$ that contains $\D$. Since $y \in W_0^{J_\co}$, for any $\a \in \Phi_{J_\co}^+$, $y \a \in \Phi^+$ and $0<(y\i(e), \a)=(e, y(\a))<1$. Hence $y \i(e) \in C$. Moreover, $$x y \i(e)=y \i \tw (e)=y \i(e+\nu_{\tw})=y \i(e)+\nu_\co.$$ Since $\<\nu_\co, \a\>=0$ for all $\a \in \Phi_J$, we have $y \i(e)$ and $y \i(e)+\nu_\co$ are contained in the same connected component of $V-\cup_k \cup_{\a \in \Phi_J} H_{\a, k}$. Hence $C \ni y \i(e)$ and $x(C) \ni x y \i(e)$ are the same connected component of $V-\cup_k \cup_{\a \in \Phi_J} H_{\a, k}$. Thus there is no hyperplane of the form $H_{\a, k}$ with $\a \in \Phi_J$ that separates $C$ from $x(C)$. So $x$ is a basic element in $\tW_J$.
\end{proof}

\subsection{}\label{wwcg} The next task of this section is to give a parametrization of straight conjugacy classes. Such parametrization coincides with the set of $\s$-conjugacy classes of $p$-adic groups \cite{H99}.

Let $P^+$ be the set of dominant coweights of $G$ and $$P^+_\QQ=\{\l \in P \otimes_\ZZ \QQ; \<\l, \a\> \ge 0 \text{ for all } \a \in \Phi^+\} \subset V.$$ Then we may identify $P^+_\QQ$ with $(P \otimes_\ZZ \QQ)/W_0$. For any $\l \in P \otimes_\ZZ \QQ$, we denote by $\bar \l$ the unique element in $P^+_\QQ$ that lies in the $W_0$-orbit of $\l$. The group $\Om'$ acts naturally on $P^+_\QQ$ and on $\tW_G/W_G \cong P/Q$. Let $\d' \in \Om'$. The map $x \mapsto (x \d'^{-1} W_G, \bar \nu_x)$ induces a natural map $$f_{\d'}: \tW_G \d' \rightarrow (P/Q)_{\d'} \times P^+_\QQ.$$
Here $\tW_G \d' \in \tW_G \backslash \tW$ is a right $\tW_G$-coset containing $\d'$ and $(P/Q)_{\d'}$ is the $\d'$-coinvariants of $P/Q$. We denote the image by $B(\tW_G, \d')$.

\begin{thm}\label{straight}
The map $f_{\d'}$ induces a bijection between the straight $\tW_G$-conjugacy classes of $\tW_G \d'$ and $B(\tW_G, \d')$.
\end{thm}
\begin{proof}
We first show that

(a) The map $f_{\d'}$ is constant on each $\tW_G$-conjugacy class.

Let $\tw=t^\chi w \d' \in \tW$ and $\tu=t^\l u\in \tW_G$, where $\chi, \l \in P$ and $w, u \in W_0$. Then $\tu \tw \tu \i=t^{\l+u \chi-(u w \d' u \i (\d') \i) \d' \l} (u w \d' u')$. Notice that for any $x \in W_0$ and $\mu \in P$, $x \mu-\mu \in Q$. Hence $t^{\l+u \chi-(u w \d' u \i (\d') \i) \d' \l} \in t^{\l+\chi-\d'(\l)} W_G$ and $$\tu \tw \tu \i \in t^{\l+\chi-\d'(\l)} W_G (u w \d' u' (\d') \i) \d'=t^{\l+u-\d'(\l)} \d' W_G.$$  Hence the images of $\tu \tw \tu\i$ and $\tw$ in $(P/Q)_{\d'}$ are the same. Assume that $n \in \NN$ and $\tw^n=t^{n \nu_{\tw}}$. Then $(\tu \tw \tu \i)^n=\tu t^{n \nu_{\tw}} \tu \i=t^\l t^{u n \nu_{\tw}} t^{-\l}=t^{u n \nu_{\tw}}$. Therefore $\nu_{\tu \tw \tu\i}=u(\nu_{\tw})$ and $\bar \nu_{\tu \tw \tu\i}=\bar \nu_{\tw}$.

(a) is proved.

Moreover, $t^{n \nu_{\tw}}=\tw t^{n \nu_{\tw}} \tw \i=t^{\chi} t^{w n \d'(\nu_{\tw})} t^{-\chi}=t^{w n \d'(\nu_{\tw})}$. Thus $\nu_{\tw}=w \d'(\nu_{\tw})$. Hence

(b) $\bar \nu_{\tw}=\d'(\bar \nu_{\tw})$ for all $\tw \in \tW_G \d'$.

By Proposition \ref{red1} and \ref{ell}, for any $\tw \in \tW$, $\tw \rightarrow u \tw_I$ for some $I \subset S$ with $W_I$ finite, $u \in W_I$, and a straight element $\tw_I$ with $\tw_I(I)=I$. By the proof of \cite[Proposition 2.2]{He2}, $f_{\d'}(\tw)=f_{\d'}(u\tw_I)=f_{\d'}(\tw_I)$. So $f_{\d'}$ is surjective.

Now we prove that $f_{\d'}$ is injective.

Let $\tw, \tw' \in \tW_G \d'$ with $f_{\d'}(\tw)=f_{\d'}(\tw')$. Assume $\tw=t^\l w \d'$ and $\tw'=t^{\l'} w' \d'$ for some $\l, \l' \in X$, $w, w' \in W_G$. Then after conjugating by a suitable element of $\tW_G$, we can assume further that $\tw W_G=\tw' W_G$.

Let $J=\{i \in S_0; \<\bar \nu_{\tw}, \a_i\>=0\}$. By (b), $\bar \nu_{\tw}=\d'(\bar \nu_{\tw})$. Then $\d'(J)=J$. By Proposition \ref{cri}, after conjugation by some elements in $W_G$, we may assume that $\tw, \tw' \in \tW_J$ and $\nu=\nu_{\tw}=\nu_{\tw'} \in P_\QQ^+$.

Let $N=\{v \in P \otimes_\ZZ \RR; (v, \a^\vee)=0 \text{ for all } \a \in \Phi\}$, where $(\, ,)$ is a $\d'$-invariant inner product on $P \otimes_\ZZ \RR$. Let $J'=S_0-J$ and $Q_J$, $Q_{J'}$ be the sublattices of $Q$ spanned by simple roots of $J$ and $J'$ respectively. Then $\d'(N)=N$ and $$P \otimes_\ZZ \RR=Q_J \otimes_\ZZ \RR \oplus Q_{J'} \otimes_\ZZ \RR \oplus N.$$

We may write $\l$ and $\l'$ as $\l=a_J+a_{J'}+c$ and $\l'=a'_J+a'_{J'}+c'$ with $a_J, a'_J \in Q_J \otimes_\ZZ \RR$, $a_{J'}, a'_{J'} \in Q_{J'} \otimes_\ZZ \RR$ and $c, c' \in N$. Since $\l-\l' \in Q$, $a_J-a'_J \in Q_J$, $a_{J'}-a'_{J'} \in Q_{J'}$ and $c=c'$.

Choose $n \in \NN$ such that $(w \d')^n=(w' \d')^n=1$. Then $$\nu=\frac{1}{n}\sum_{k=0}^{n-1}(w\d')^k(\l) \in \frac{1}{n}\sum_{k=0}^{n-1}(\d')^k(a_{J'})+Q_J \otimes_\ZZ \QQ + N.$$ Similarly, $$\nu \in \frac{1}{n}\sum_{k=0}^{n-1}(\d')^k(a'_{J'})+Q_J \otimes_\ZZ \QQ + N.$$ Hence $$\sum_{k=0}^{n-1}(\d')^k(a_{J'}-a'_{J'})=0.$$ Since $a_{J'}-a'_{J'} \in Q_{J'}$, then $a_{J'}-a'_{J'}=\th-\d'(\th)$ for some $\th \in Q_{J'}$. Let $\tw''=t^\th \tw' t^{-\th}$. By Condition (2) of Theorem \ref{cri}, $\tw'$ and $\tw''$ are conjugate to a basic elements in $\tW_J$ by elements in $Q_J \rtimes W_J$. Moreover, $\l \in \l'+\th-\d'(\th)+Q_J$ and $(Q_J \rtimes W_J) \tw=(Q_J \rtimes W_J)\tw'' \in (Q_J \rtimes W_J) \backslash \tW_J$. Thus $\tw$ and $\tw''$ are conjugate to the same basic element of $\tW_J$ by an element in $Q_J \rtimes W_J$ and $\tw$ and $\tw'$ are in the same $\tW_G$-conjugacy class.
\end{proof}

\subsection{} Let $\t \in \Om$. Conjugation by $\t$ gives a permutation on the set of affine simple reflections $S$ of $W$.  We call that $\t$ is {\it superbasic} if each orbit is a union of connected components of the Dynkin diagram of $S$.

In this case, any two vertices in the same connected components of $S$ have the same numbers of edges and thus $S$ is a union of affine Dynkin diagrams
of type $\tilde A$. Hence it is easy to see that $\t \in \Om$ is a superbasic element of $\tW$ if and only if $W=W_1^{m_1} \times \cdots \times W_l^{m_l}$,
where $W_i$ is an affine Weyl group of type $\tilde A_{n_i-1}$ and $\t$ gives an order $n_i m_i$ permutation on $W_i^{m_i}$. 

\subsection{} We follow the notation in $\S$\ref{wwcg}. Let $\d' \in \Om'$. Then any fiber of the map $f_{\d'}: \tW_G \d' \rightarrow (P/Q)_{\d'} \times P^+_\QQ$ is a union of $\tW_G$-conjugacy classes. We call a $\tW_G$-conjugacy class in $\tW_G \d'$ {\it superstraight} if it is a fiber of $f_{\d'}$. By Theorem \ref{straight}, any fiber contains a straight $\tW_G$-conjugacy class. Hence a superstraight conjugacy class is in particular straight. Now we give a description of superstraight $\tW_G$-conjugacy classes which is analogous to Proposition \ref{cri}.

\begin{prop}\label{superstraight}
We keep notations in $\S$\ref{wgo}. Let $\co$ be a $\tW_G$-conjugacy class of $\tW$. Then the following are equivalent:

(1) $\co$ is a superstraight.

(2) For some (or, equivalently any) $\tw \in \co$, $H \cap V_{\tw}=\emptyset$ for any $H \in \fH(\nu_{\tw})$. Here $H \in \fH(\nu_{\tw})=\{H_{\a, k} \in \fH; \<\nu_{\tw}, \a\>=0, k \in \ZZ\}$.

(3) There exists a superbasic element $x$ in $\tW_{J_{\co}}$ and $y \in W_G^{J_\co}$ such that $\nu_x=\nu_\co$ and $y x y\i \in \co_{\min}$.
\end{prop}
\begin{proof}
We assume that $\co \subset \tW_G \d'$ for $\d' \in \Om'$.

(1)$\Rightarrow$(3). By Proposition \ref{cri}, there exists a basic element $x$ in $\tW_{J_{\co}}$ and $y \in W_G^{J_\co}$ such that $\nu_x=\nu_\co$ and $y x y\i \in \co_{\min}$. Assume that $x$ is not superbasic in $\tW_{J_{\co}}$. Then there exists an $x$-orbit $O$ such that $C \cap O \subsetneq C$ for each connected component $C$ of the Dynkin diagram $\tW_{J_\co}$.

Note that $C \cap O \subsetneq C$ is the Dynkin diagram of a finite Weyl group. Hence $W_O$ is a finite product of Weyl groups corresponding to $C \cap O$ and hence is finite. By the proof of \cite[Proposition 2.2]{He2}, $f_{\d'}(w x)=f_{\d'} (x)$ for all $w \in W_O$. In particular, $s_j x$ and $x$ are in the same fiber of $f_{\d'}$ for any $j \in O$.

However, $\ell(s_j x) \equiv \ell(x)+1 \mod 2$. Thus $s_j x$ and $x$ are not in the same conjugacy class. So $\co$ is not superstraight.



(2)$\Rightarrow$(1).
Let $\co'$ be a $\tW_G$-conjugacy class such that $\co'$ and $\co$ are in the same fiber of $f_{\d'}$. By Proposition \ref{red1} and Proposition \ref{ell}, $\co'$ contains an element of the form $u x$, where $x$ is straight, $f_{\d'}(u x)=f_{\d'}(x)$ and $u \in W_{V_{u x}}$. By Theorem \ref{straight}, $x \in \co$. By the proof of \cite[Proposition 2.2]{He2}, $\nu_{u x}=\nu_x$.

Let $v \in V_{u x}$. By Lemma \ref{plusv}, $u x(v)=v+\nu_{u x}=v+\nu_x \in V_{u x}$. Since $u \in W_{V_{ux}}$, $x(v)=u \i u x(v)=u x(v)=v+\nu_x$ and $v \in V_x$. Thus $V_{u x} \subset V_x$.

Let $H=H_{\a, k} \in \fH_{V_{u x}}$. Then for any $v \in V_{u x}$, $\<v, \a\>=k$. In particular $\<v+\nu_{x}, \a\>=\<v, \a\>+\<\nu_{x}, \a\>=k$ and $\<\nu_{x}, \a\>=0$. However, by our assumption $H \cap V_{x}=\emptyset$. Thus $H \cap V_{u x}=\emptyset$ and $\fH_{V_{u x}}=\emptyset$. Hence $W_{V_{u x}}=\{1\}$ and $u=1$. So $\co'=\co$.

(3)$\Rightarrow$(2). Let $C$ be the unique connected component of $V-\cup_{H \in \fH(\nu_\co)}H$ containing $\D$. We call $H \in \fH(\nu_\co)$ a wall of $C$ if $H \cap \bar C$ spans $H$. Let $\fH(C)$ be the set of all walls of $C$. Since $x$ is superbasic, $x C=C$ and $x$ acts transitively on $\fH(C)$. Since $C$ is convex, $V_x \cap C\neq \emptyset$.

Suppose that $V_x \cap H' \neq \emptyset$ for some $H' \in \fH(\nu_\co)$. Let $p \in V_x \cap H'$ and $q \in V_x \cap C$. Then the affine line $L(p, q) \subset V_x$ intersects with $\partial \bar C \subset \cup_{H \in \fH(C)}H$. Let $v \in L(p, q) \cap \partial \bar C$. Then $v \in H_0$ for some $H_0 \in \fH(C)$. Thus $x^m(v)=v+m\nu_\co \in x^m(H_0)$ for $m \in \ZZ$. Notice that $\RR \nu_\co+H=H$ for $H \in \fH(\nu_\co)$. Thus $v \in x^m(H_0)$ for all $m \in \ZZ$. As $x$ acts transitively on $\fH(C)$, we have $v \in \cap_{H \in \fH(C)} H$. However, $\cap_{H \in \fH(C)} H=\emptyset$. That is a contradiction.
\end{proof}

\subsection{} In the rest of this section, we'll show that any two straight elements in the same conjugacy class are conjugate by cyclic shift, which is analogous to $\S$0.1 (3) for an elliptic conjugacy class of a finite Coxeter group.

In order to do this, we use the following length formula. The proof is similar to \cite[Proposition 2.3]{HN} and is omitted here.


\begin{prop}\label{con}
Let $\tw \in \tW$ and $K \subset V_{\tw}$ be an affine subspace with $\tw(K)=K$. Let $A$ and $A'$ be two alcoves in the same connected component of $V-\cup_{H \in \fH_K}H$. Assume that $\bar A \cap \bar A' \cap K$ spans a codimension $1$ subspace of $K$ of the form $H_0 \cap K$ for some $H_0 \in \fH$ and $\tw(H_0 \cap K) \neq H_0 \cap K$. Then $$\ell(\tw_A)=\ell(\tw_{A'})=\<\bar \nu_{\tw}, 2\rho\>+\sharp\fH_K(A, \tw A).$$
\end{prop}

\subsection{} Let $K \subset V$ be an affine subspace. We call an affine subspace  of $K$ of the form $H \cap K$ with $H \in \fH-\fH_K$ a {\it $K$-hyperplane} and a connected component of $K-\cup_{H \in \fH-\fH_K}(H \cap K)$ a {\it $K$-alcove}.

Let $A_K$ be a $K$-alcove and $H \in \fH-\fH_K$. If the set of inner points $H_{A_K}=(H \cap \bar A_K)^0 \subset H \cap \bar A_K$ spans $H \cap K$, then we call $H \cap K$ a {\it wall} of $A_K$ and $H_{A_K}$ a {\it face} of $A_K$.

\begin{lem}\label{K-alcove}
Let $K \subset V$ be an affine subspace. Let $A_K$ be a $K$-alcove and $C$ be a connected component of $V-\cup_{H \in \fH_K}H$. Then there exists a unique alcove $A$ in $C$ such that $A_K \subset \bar A$. Moreover, $\bar A_K=\bar A \cap K$.
\end{lem}
\begin{proof}
Choose $e \in A_K$. Then $e$ is a regular point of $K$. Hence there is a unique alcove $A$ in $C$ such that $e \in \bar A$. We show that

(a) $A_K \subset \bar A$.

Let $e' \in A_K$ and $A'$ be the unique alcove in $C$ with $e' \in \bar {A'}$. Since $e, e' \in A_K$ and $A_K$ is convex, we see $e$ and $e'$ are in the same connected component of $V-H$ for any $H \in \fH-\fH_K$. Hence $A$ and $A'$ are in the same connected component of $V-\cup_{H \in \fH-\fH_K}H$. Since $A$ and $A'$ are in the same connected component $C$ of $V-\cup_{H \in \fH_K}H$, they are in the same connected component of $V-\cup_{H \in \fH}H$. Hence $A=A'$.

(a) is proved.

For $H \in \fH$, let $C_H$ be the connected component of $V-H$ that contains $A$. For $H \in \fH-\fH_K$, let $C_{H \cap K}$ be the connected component of $K-H \cap K$ that contains $A_K$. Then $\bar C_H \cap K=\begin{cases} C_{H \cap K}, & \text{ if } H \in \fH-\fH_K \\ K, & \text{ if } H \in \fH_K \end{cases}$. Note that $$A=\cap_{H \in \fH} C_H, \qquad A_K=\cap_{H \in \fH-\fH_K} C_{H \cap K}.$$ Then $$\bar A_K=\cap_{H \in \fH-\fH_K} \bar C_{H \cap K}=\cap_{H \in \fH} (\bar C_H \cap K)=\bar A \cap K.$$
\end{proof}

\begin{lem}\label{sequence}
Let $K \subset V$ be an affine subspace. Let $A$ and $A'$ be two distinct alcoves in the same connected component $C$ of $V-\cup_{H \in \fH_K}H$ and $\bar A$, $\bar A'$ contain regular points of $K$. Then there is a sequence of alcoves $A=A_0, A_1, \cdots A_r=A'$ in $C$ such that $\bar A_i$ contains a regular point of $K$ and $\bar A_{i-1} \cap \bar A_i \cap K$ spans a codimension one affine subspace of $K$ for $1 \le i \le r$.
\end{lem}
\begin{proof}
Let $A_K$ (resp. $A_K'$) be the unique $K$-alcove contained in $\bar A$ (resp. $\bar A'$). Since $A \neq A' \subset C$, we see that $A_K \neq A_K'$. Then there is a sequence of $K$-alcoves $A_K=A_K^0, A_K^1, \cdots, A_K^r=A_K'$ such that $A_K^{i-1}$ and $A_K^i$ have a common face for any $i$. By Lemma \ref{K-alcove}, there is a unique alcove $A_i$ in $C$ such that $\bar A_K^i=\bar A_i \cap K$ for each $i$. It is clear that the sequence $A_0, A_1, \cdots, A_r$ meets our requirement.
\end{proof}

\begin{lem}\label{identity}
Let $\tw \in \tW$. Let $K \subset V_{\tw}$ be an affine subspace such that $\tw K=K$. Let $A$ and $A'$ be two alcoves such that $\bar A \cap \bar A'$ contains a regular point of $K$ and $\tw_A, \tw_{A'}$ are straight elements. Then $\tw_A=\tw_{A'}$.
\end{lem}
\begin{proof}
We may assume that $A$ is the fundamental alcove $\D$ by replacing $\tw$ by $\tw_A$. We simply write $I$ for $I(K, \D)$. By Proposition \ref{ell}, $\tw \in {}^I \tW^I$ and $\tw I=I$. Since $\bar A' \cap \bar \D$ contains a regular point of $K$, $x_{A'} \in W_I$. Thus $$\tw_{A'}=x_{A'} \i \tw x_{A'}=(x_{A'} \i \tw x_{A'} \tw \i) \tw$$ and $x_{A'} \i \tw x_{A'} \tw \i \in W_I$. Therefore $\ell(\tw_{A'})=\ell(\tw)+\ell(x_{A'} \i \tw x_{A'} \tw \i)$. Since $\ell(\tw_{A'})=\ell(\tw)$, we have
$x_{A'} \i \tw x_{A'} \tw \i=1$ and $\tw_{A'}=\tw$.
\end{proof}
\begin{thm}\label{cyclic}
Let $\co$ be a straight $W$-conjugacy class of $\tW$. Then for any $\tw, \tw' \in \co_{\min}$, $\tw \approx \tw'$.
\end{thm}
\begin{proof}
Let $\tu \in \co$ and $K=V_{\tu}$. Then by Proposition \ref{red1} and Proposition \ref{ell}, we may assume that $\tw=\tu_A$ and $\tw'=\tu_{A'}$, where $A$ and $A''$ are two alcoves whose closures contain regular points of $K$. Let $C$ be the connected component of $V-\cup_{H \in \fH_K}H$ that contains $A$ and $A''$ be the unique alcove in $C$ such that $\bar A' \cap \bar A''$ contains a regular point of $K$. By Proposition \ref{ell}, we have $\tu_{A''} \in \co_{\min}$. By Lemma \ref{identity}, we have that $\tu_{A'}=\tu_{A''}$.

It remains to show that $\tu_A \approx \tu_{A''}$. Assume $A \neq A''$. By Lemma \ref{sequence}, there is a sequence of alcoves $A=A_0, A_1, \cdots A_r=A''$ in $C$ such that $A_i$ contains a regular point of $K$ and $\bar A_{i-1} \cap \bar A_i \cap K$ spans a codimension one affine subspace $P_i$ of $K$ for $i=1, 2, \cdots, r$. By Proposition \ref{ell}, $\tu_{A_i} \in \co_{\min}$ for any $i$.

If $\tu P_i=P_i$, by Lemma \ref{identity}, we have that $\tu_{A_{i-1}}=\tu_{A_i}$. If $\tu(P_i) \neq P_i$, then there is a sequence of alcoves $A_{i-1}=B_0, B_1, \cdots, B_s=A_i$ in $C$ such that $B_{k-1}$ and $B_k$ share a common face and $\bar B_{k-1} \cap \bar B_k \cap K$ spans $P_i$ for $k=1, \cdots, s$. By Proposition \ref{con}, we have that $\ell(\tu_{B_{k-1}})=\ell(\tu_{B_k})$ and $\tu_{B_{k-1}} \approx \tu_{B_k}$ for $k=1, \cdots, s$. So $\tu_{A_{i-1}} \approx \tu_{A_i}$. Hence $\tu_A \approx \tu_{A''}$.

\end{proof}

\section{Centralizer in $W$}

\subsection{}\label{nice} Let $(W, S)$ be a Coxeter group and $\Om$ be a group with a group homomorphism to $\Aut(W, S)$. Let $\tW=W \rtimes \Om$. Let $\tw \in \tW$ be a minimal length element in its conjugacy class. Let $\cp_{\tw}$ be the set of sequences $\textbf{i}=(s_1, \cdots, s_r)$ of $S$ such that $$\tw \overset {s_1} \to s_1 \tw s_1 \overset {s_2} \to \cdots \overset {s_r} \to s_r \cdots s_1 \tw s_1 \cdots s_r.$$ We call such sequence a path form $\tw$ to $s_r \cdots s_1 \tw s_1 \cdots s_r$. Denote by $\cp_{\tw,\tw}$ the set of all paths from $\tw$ to itself. Let $W_{\tw}=\{x \in W; \ell(x \i \tw x)=\ell(\tw)\}$ and $Z(\tw)=\{x \in W; x \tw=\tw x\}$.

There is a natural map $$\tau_{\tw}: \cp_{\tw} \to W_{\tw},~(s_1, \cdots, s_r) \mapsto s_1 \cdots s_r.$$ which induces a natural map $\tau_{\tw,\tw}: \cp_{\tw,\tw} \to Z(\tw)$.

We call a $W$-conjugacy class $\co$ of $\tW$ {\it nice} if for some (or equivalently, any) $\tw \in \co_{\min}$, the map $\tau_{\tw}: \cp_{\tw} \to W_{\tw}$ is surjective. It is easy to see that $\co$ is nice if and only if
properties (1) and (2) below hold for $\co$:

(1) For any $\tw, \tw' \in \co_{\min}$, $\tw \approx \tw'$;

(2) For any $\tw \in \co_{\min}$, the map $\tau_{\tw,\tw}: \cp_{\tw,\tw} \to Z(\tw)$ is surjective.

The definition of nice conjugacy classes is inspired by a conjecture of Lusztig \cite[1.2]{L4} that property (2) holds for elliptic conjugacy classes of a finite Weyl group.

\subsection{} For finite Weyl groups, nice conjugacy classes play an important role in the study of Deligne-Lusztig varieties and representations of finite groups of Lie type. Property (1) is a key ingredient to prove that Deligne-Lusztig varieties corresponding to minimal length elements in a nice conjugacy class are universally homeomorphic. Property (2) leads to nontrivial (quasi-)automorphisms on Deligne-Lusztig varieties and their cohomology groups. For more details, see \cite{DM} and \cite{L4}.

Nice conjugacy classes for affine Weyl groups will also play an important role in the study of affine Deligne-Lusztig varieties. See \cite{H99}.

\subsection{} The main goal of this section is to classify nice conjugacy classes for both finite and affine Weyl groups.

We first consider finite Coxeter group. Let $(W_0, S_0)$ be a finite Coxeter group and $\Om' \subset \Aut(W_0, S_0)$. Set $\tW_0=W_0 \rtimes \Om'$. Let $z=w \s \in \tW_0$ with $w \in W_0$ and $\s \in \Om'$. We denote by $\supp(w)$ the {\it support} of $w$, i.e., the set of simple reflections appearing in a reduced expression of $w$. Set $\supp(z)=\cup_{n \in \ZZ} \s^n \supp(w)$. We call $\supp(z)$ the support of $z$. It is a $\s$-stable subset of $S_0$.

We call a $W_0$-conjugacy class $\co$ of $\tW_0$ {\it elliptic} if $\supp(\tw)=S_0$ for any $\tw \in \co$. An element in an elliptic conjugacy class is called an elliptic element.

We have the following result.

\begin{thm}
Any elliptic conjugacy class in a finite Coxeter group is nice.
\end{thm}


\subsection{}\label{ccon} In order to classify nice conjugacy classes for finite Coxeter group, we first recall the geometric interpretation of conjugacy classes and length function in \cite{HN}.

Let $V$ be a finite dimensional vector space over $\RR$ and $\fH_0$ be a finite set of hyperspaces of $V$ through the origin such that $s_H(\fH_0)=\fH_0$ for all $H \in \fH_0$. Let $W_0 \subset GL(V)$ be the subgroup generated by $s_H$ for $H \in \fH_0$. Let $\fC(\fH_0)$ be the set of connected components of $V-\cup_{H \in \fH_0} H$. We call an element in $\fC(\fH_0)$ a {\it chamber}. We fix a fundamental chamber $C_0$. For any two chambers $C, C'$, we denote by $\fH_0(C,C')$ the set of hyperspaces in $\fH_0$ separating $C$ from $C'$. Let $S_0=\{s_H \in W_0; \sharp \fH_0(C_0,s_H C_0)=1\}$. Then $(W_0, S_0)$ is a finite Coxeter group. Let $\tW_0=W_0\rtimes \Om'$ where $\Om' \subset GL(V)$ consists of automorphism preserving $S_0$. Then $\ell(w)=\sharp \fH_0(C_0,wC_0)$ is the length function on $\tW_0$.

It is known that $W_0$ acts simply transitively on the set of chambers. For any chamber $C$, we denote by $x_C$ the unique element in $W_0$ with $x_C(C_0)=C$. Here $C_0$ is the fundamental chamber. Then any element in the $W_0$-conjugacy class of $\tw$ is of the form $\tw_C=x_C \i \tw x_C$ for some chamber $C$.

For any $\tw \in \tW_0$, we denote by $\fC_{\tw}(\fH_0)$ the set of chambers $C$ such that $\tw_C$ is of minimal length in its $W_0$-conjugacy class. We denote by $V_{\fH_0}^{subreg}$ the set of points in $V$ that is contained in at most one $H$ for $H \in \fH_0$. By \cite[Lemma 4.1]{HN},

(a) A $W_0$-conjugacy class $\co$ of $\tW_0$ is nice if and only if $(\cup_{A \in \fC_{\tw}(\fH_0)} \bar A) \cap V_{\fH_0}^{subreg}$ is connected for some (or equivalently, any) $\tw \in \co$.

By \cite[Lemma 7.2]{He1}, a $W_0$-conjugacy class $\co$ of $\tW_0$ is elliptic if and only if for some (or equivalently, any) element $\tw \in \co$, the fixed point set $V^{\tw} \subset V^{W_0}$. Now we introduce the weakly elliptic conjugacy classes.

\begin{prop}\label{well}
Let $\co$ be a $W_0$-conjugacy class of $\tW_0$. Then the following conditions are equivalent:

(1) For some $\tw \in \co_{\min}$, $\supp(\tw)$ is a union of some connected components of the Dynkin diagram of $S_0$ and $\tw$ commutes with any element in $S_0-\supp(\tw)$.

(2) For any $\tw \in \co$, $\supp(\tw)$ is a union of some connected components of the Dynkin diagram of $S_0$ and $\tw$ commutes with any element in $S_0-\supp(\tw)$.

(3) For some (or equivalently, any) element $\tw \in \co$, $s_H V^{\tw}=V^{\tw}$ for all $H \in \fH_0$.

In this case, we call $\co$ a {\it weakly elliptic conjugacy class} and any element in $\co$ a {\it weakly elliptic element}.
\end{prop}

\begin{proof}
(1)$\Rightarrow$(2). Assume that $\co \subset W_0 \s$ for $\s \in \Om'$. Assume that $\tw \in \co_{\min}$ and $\supp(\tw)=S'_0$ is a union of some connected components of the Dynkin diagram of $S_0$, $\s(s)=s$ for all $s \in S_0-S'_0$. Then $\supp(x \tw x \i) \subset S'_0$ for any $x \in W_0$. By our assumption, $\tw$ is elliptic in $W_{S'_0} \rtimes \<\s\>$. Thus $\supp(x \tw x \i)=S'_0$ for any $x \in W_0$.

(2)$\Rightarrow$(3). Assume that $\co \subset W_0 \s$ for some $\s \in \Om'$. By assumption, there exists $S'_0 \subset S_0$, which is a union of some connected components of the Dynkin diagram of $S_0$ such that $\s(s)=s$ for all $s \in S_0 - S'_0$ and $\supp(\tw)=S'_0$ for any $\tw \in \co$. Then $\co$ is elliptic in $W_{S'_0} \rtimes \<\s\>$. Therefore $V^{\tw}=V^{W_{S'_0} \rtimes \<\s\>}$ for all $\tw \in \co$. Hence $s_H V^{\tw}=V^{\tw}$ for all $H \in \fH_0$.

(3)$\Rightarrow$(1). If $V^{\tw}=\{0\}$ for some $\tw \in \co$, then $\co$ is elliptic and condition (2) is satisfied. Now we assume that $V^{\tw} \neq \{0\}$ for any $\tw \in \co$. By \cite[Proposition 2.2]{HN}, there exists $\tw \in \co_{\min}$ such that $\bar C_0$ contains a regular point $e$ of $V^{\tw}$. Assume that $\tw=w \s$ for $w \in W_0$ and $\s \in \Om'$. Then $w \s(e)=e$ is the unique dominant element in the $W_0$-orbit of $e$. Thus $e=\s(e)=w(e)$.

Let $J=\{s \in S_0; s(e)=e\}$. Then $\s(J)=J$ and $w \in W_J$. Since $e$ is a regular element in $V^{\tw}$, we have that $V^{\tw} \subset V^{W_J}$. Thus $\tw$ is an elliptic element in $W_J \rtimes \<\s\>$ and $\supp(\tw)=J$.

Let $s \in S_0 - J$. Then $s(V^{\tw})=V^{\tw}$. Thus $s(e)=e+\<e, \a\> \a^\vee \in V^{\tw}$, where $\a$ is the positive root corresponding to $s$. Since $s \notin J$, we have $s(e) \neq e$ and $\a^\vee \in V^{\tw}$. Hence $\tw s=s \tw$. The statement follows from the following Lemma \ref{commute}.
\end{proof}

\begin{lem}\label{commute}
Let $\tw=w \s$ with $w \in W_0$ and $\s \in \Om'$. Let $s \in S_0-\supp(\tw)$. Then $s \tw=\tw s$ if and only if $s=\s(s)$ commutes with each element of $\supp(\tw)$.
\end{lem}
\begin{proof}
Let $J=\supp(w)$ and $J'=\{s' \in J; ss'=s's\}$. Write $w$ as $w=a b c$ for $a \in W_{J'}$, $c \in W_{\s(J')}$ and $b \in {}^{J'} W_0^{\s(J')}$. Then $s b=b \s(s)$. Since $s \notin J$ and each element $J-J'$ does not commute with $s$, we have that $b \s(s)=s b \in {}^{J-J'} W_0$. Therefore $b \in {}^{J-J'} W_0$. Since $b \in {}^{J'} W_0$ and $b \in W_J$, we must have that $b=1$. Hence $s=\s(s)$ and $\s(J')=J'$. So $J'=\supp(w)=J$. Now $s s'=s' s$ for all $s' \in J$. Hence for any $n \in \ZZ$ and $s' \in J$, $s \s^n(s')=\s^n(s s')=\s^n(s' s)=\s^n(s') s$. Therefore $s$ commutes with every element in $\supp(\tw)=\cup_{n \in \ZZ} \s^n(J)$.
\end{proof}

\

Now we classify nice conjugacy classes for $\tW_0$.

\begin{thm}\label{lusztig}
Let $\co$ be a $W_0$-conjugacy class of $\tW_0$. Then $\co$ is nice if and only if it is weakly elliptic.
\end{thm}

\begin{proof}
Suppose that $\co$ is nice. Let $\tw \in \co_{\min}$ and $S'_0=\supp(\tw)$. Then $w_0 \tw w_0 \in \co_{\min}$, where $w_0$ is the largest element of $W_0$. Hence there is a reduced expression $w_0=s_1 s_2 \cdots s_r$ such that $\tw_0 \approx \tw_1 \approx \cdots \approx \tw_r$, where $\tw_i=(s_1 \cdots s_i) \i \tw (s_1 \cdots s_i)$. By \cite[Lemma 7.4]{He1}, $\supp(\tw_i)=\supp(\tw)=S'_0$ for all $i$. Therefore $s_{j+1} \tw_j s_{j+1}=\tw_j$ if $s_{j+1} \notin S'_0$. By Lemma \ref{commute}, $s t=t s$ and $\tw t=t \tw$ for all $s \in S'_0$ and $t \in S_0-S'_0$. Hence $\co$ is weakly elliptic.

Suppose that $\co$ is weakly elliptic. Let $\tw \in \co_{\min}$ and $J=\supp(\tw)$. Then $J$ and $S_0-J$ are unions of connected components of the Dynkin diagram of $S_0$. In particular, any simple reflection in $J$ commutes with simple reflection in $S_0-J$. Let $x \in W_0$ such that $\ell(x \tw x \i)=\ell(\tw)$. We may write $x$ as $x_1 x_2$ for $x_1 \in W_J$ and $x_2 \in S_0-J$. Then $x \tw x \i=x_1 (x_2 \tw x_2 \i) x_1 \i=x_1 \tw x_1 \i$. By definition, $\tw$ is elliptic in $W_J \rtimes\<\s\>$ for some $\s \in \Om'$. Thus by \cite[Corollary 4.4]{HN}, $x_1$ is in the image of $\t_{\tw}$. Since any simple reflection in $S_0-J$ commutes with $\tw$, $x_2$ is also in the image of $\t_{\tw}$. Thus $\co$ is nice.
\end{proof}

\subsection{}\label{acon} Now we study affine Weyl groups. We keep the notation in $\S$\ref{wgo}. Let $\tW_0=W_0 \rtimes \Om'$. Then $\tW=W_G \rtimes \Om=P \rtimes \tW_0$. For any $\tw \in \tW$, we denote by $\fA_{\tw}$ the set of alcoves $A$ such that $\tw_A$ is of minimal length in its $W$-conjugacy class. We denote by $V^{subreg}$ the set of points in $V$ that is contained in at most one $H$ for $H \in \fH$.  Similar to the proof of \cite[Lemma 4.1]{HN},

(a) A $W$-conjugacy class $\co$ of $\tW$ is nice if and only if $(\cup_{A \in \fA_{\tw}} \bar A) \cap V^{subreg}$ is connected for some (or equivalently, any) $\tw \in \co$.

\subsection{} Let $y \in \tW$. Let $K \subset V_y$ be an affine subspace such that $y K=K$. Choose $p \in K$. Define $\bar y= T_{-\nu_y-p} \circ y \circ T_p \in GL(V)$, where $T_v$ denotes the map of translation by $v$ for $v \in V$. Then $\bar y(0)=0$ and $\bar y$ is the image of $y$ under the map $\tW=P \rtimes \tW_0 \to \tW_0$. In other words, $\bar y$ is the finite part of $y$.

Set $\fH_{K,p}=T_{-p}(\fH_K)$. Then any element in $\fH_{K, p}$ is a hyperplane through $0$ and contains $V^{\bar y}=T_{-p} V_y$. Since $y$ preserve $\fH_K$, $\bar y$ preserves $\fH_{K, p}$.

The following result relates $\S$\ref{ccon} (a) with $\S$\ref{acon} (a).

\begin{lem}\label{key11}
Keep notations as above. Assume $p$ is a regular point of $K$ and $A, A' \in \fA_y$ with $p \in \bar A \cap \bar A'$. Let $C$ (resp. $C'$) be the unique element of $\fC(\fH_{K,p})$ such that $A \subset T_p(C)$ (resp. $A \subset T_p(C')$). Then $A$ and $A'$ are in the same connected component of $(\cup_{A \in \fA_y}\bar A) \cap V^{subreg}$ if and only if $C$ and $C'$ are in the same connected component of  $(\cup_{C \in \fC_{\bar y}(\fH_{K,p})}\bar C) \cap V_{\fH_{K,p}}^{subreg}$.
\end{lem}
\begin{proof}
Suppose that $C$ and $C'$ are in the same connected component of  $(\cup_{C \in \fC_{\bar y}(\fH_{K,p})}\bar C) \cap V_{\fH_{K,p}}^{subreg}$. Then there is a sequence $C=C_1, C_2, \cdots, C_t=C'$ in $\fC_{\bar y}(\fH_{K,p})$ such that $\bar C_i \cap \bar C_{i+1}$ spans $H_i' \in \fH_{K,p}$ for $i=0,\cdots,t-1$. Note that all the numbers $\sharp \fH_{K,p}(C_i, \bar yC_i)$ are the same. Let $A_i$ be the unique alcove in $T_p(C_i)$ whose closure contains $p$. Then $\bar A_i \cap \bar A_{i+1}$ spans $T_p(H_i') \in \fH_K$ for each $i$. By Proposition \ref{ell}, \begin{align*} \sharp \fH(A_i, yA_i) &=\sharp \fH_K(A_i,yA_i)+\<\bar \nu_y, 2\rho\>=\sharp \fH_{K,p}(C_i, \bar y C_i)+\<\bar \nu_y, 2\rho\>\\ &=\sharp \fH(A, y A).\end{align*} Hence $A_i \in \fA_y$ for all $i$ and all the $A_i$ are in the same connected component of $(\cup_{A \in \fA_y}\bar A) \cap V^{Subreg}$.

Suppose that $A$ and $A'$ are in the same connected component of $(\cup_{A \in \fA_y}\bar A) \cap V^{subreg}$. There is a sequence of alcoves $A=A_1, A_2, \cdots, A_r=A'$ in $\fA_y$ such that $A_i$ and $A_{i+1}$ share a common face which spans $H_i \in \fH$ for all $i$. By Proposition \ref{ell}, $C, C' \in \fC_{\bar y}(\fH_{K,p})$. Now we define a sequence of chambers in $\fC_{\bar y}(\fH_{K,p})$ as follows.  Let $C_1=C$. Assume that $C_i$ is already defined for $i \geq 1$. Let $j_i=\max \{k; T_{-p}(A_k) \subset C_i\}$. Let $C_{i+1}$ be the unique connected component containing $T_{-p}(A_{j_i+1})$. We obtain a sequence $C=C_1, \cdots, C_s=C'$ in this way.

Notice that $$\fH(A_{j_i}, y A_{j_i})-\{H_{j_i},y H_{j_i}\} \subset \fH(A_{j_i+1}, y A_{j_i+1}) \subset \fH(A_{j_i}, y A_{j_i}) \cup \{H_{j_i},y H_{j_i}\}.$$ Since $A_{j_i}, A_{j_i+1} \in \fA_y$, $\sharp \fH(A_{j_i}, y A_{j_i})=\sharp \fH(A_{j_i+1}, y A_{j_i+1})$ and $\{H_{j_i},y H_{j_i}\} \cap \fH(A_{j_i+1}, y A_{j_i+1})$ consists of at most one element. Hence $$\{T_{-p}(H_{j_i}),T_{-p}(y H_{j_i})=\bar y(T_{-p}(H_{j_i}))\} \cap \fH_{K,p}(C_{i+1}, \bar y C_{i+1})$$ consists of at most one element. Notice that $\sharp \fH_{K,p}(C_1, \bar y C_1)$ is minimal among all the chambers in $\fC(\fH_{K, p})$. Thus $$\sharp \fH_{K, p}(C_1, \bar y C_1)=\sharp \fH_{K, p}(C_2, \bar y C_2)=\cdots=\sharp \fH_{K, p}(C_s, \bar y C_s)$$ and $C_1, \cdots, C_s \in \fC_{\bar y}(\fH_{K,p})$. By our construction $\bar C_i \cap \bar C_{i+1}$ spans $T_{-p}(H_{j_i})$ for $i=0, \cdots, s-1$. So $C_1, \cdots, C_s$ are in the same connected component of $(\cup_{C \in \fC_{\bar y}(\fH_{K,p})}\bar C) \cap V_{\fH_{K,p}}^{subreg}$.
\end{proof}

\

Now we classify nice conjugacy classes for affine Weyl groups.

\begin{thm}\label{affinenice}
Let $\co$ be a $W$-conjugacy class of $\tW$ and $\co \subset \tW_G \s$ for $\s \in \Om'$. Then the following conditions are equivalent:

(1) $\co$ is nice.

(2) For some (or equivalently, any) $y \in \co$, $s_H(V_y)=V_y$ for any $H \in \fH(\nu_y)$ with $H \cap V_y \neq \emptyset$.

(3) For some (or equivalently, any) $y \in \co$ with $\nu_y=\nu_\co$, $\bar y$ is a weakly elliptic element in $W_{J_\co} \rtimes\<\s\>$.\end{thm}
\begin{proof}
(3) $\Rightarrow$ (2). Let $y \in \co$ with $\nu_y=\nu_\co$. Let $p \in H \cap V_y$ and $H'=T_{-p} H$. Then $H' \in \fH_{V_y, p}$. Hence $s_{H'}(V^{\bar y})=V^{\bar y}$. Since $V_y=T_p(V^{\bar y})$, $s_H(V_y)=V_y$.

(2) $\Rightarrow$ (3). Let $H=H_{\a, 0}$ with $\<\nu_y, \a\>=0$. If $V^{\bar y} \subset H$, then $s_H(V^{\bar y})=V^{\bar y}$. Otherwise, $H \cap T_v(V^{\bar y}) \neq \emptyset$ for all $v \in V$. In particular, let $p \in V_y$, then $H$ intersects $V_y=T_p(V^{\bar y})$. By condition (2), $s_H(V_y)=V_y$. Hence $s_H(V^{\bar y})=V^{\bar y}$.

(1) $\Rightarrow$ (2). Let $H \in \fH(\nu_y)$ such that $H \cap V_y \neq \emptyset$. Set $K=H \cap V_y$. Then $K$ is an affine subspace of $V_y$ of codimension at most $1$ and $y K=K$. By Proposition \ref{ell} and Lemma \ref{K-alcove}, there exists an alcove $A \in \fA_y$ whose closure contains a regular point $p$ of $K$. By condition (1) and Lemma \ref{key11}, $(\cup_{C \in \fC_{\bar y}(\fH_{K,p})}\bar C) \cap V_{\fH_{K,p}}^{subreg}$ is connected. By Theorem \ref{lusztig}, we have $s_{T_{-p}(H)}(V^{\bar y})=V^{\bar y}$, that is, $s_H(V_y)=V_y$.

(2) $\Rightarrow (1)$. Let $y \in \co$ with $\nu_y=\nu_\co$. By Proposition \ref{red1}, it suffices to prove the following statement:

{\it Let $A, A' \in \fA_y$ such that $\bar A$ and $\bar A'$ contain regular points of $V_y$. Then $A$ and $A'$ are in the same connected component of $(\cup_{A \in \fA_y}\bar A) \cap V^{subreg}$.}

Let $C$ be the connected component of $V-\cup_{H \in \fH_{V_y}} H$ that contains $A$ and $A''$ be the unique alcove in $C$ such that $\bar A' \cap \bar A''$ contains a regular point $q$ of $V_y$. By condition (3) and $\S$\ref{ccon} (a), $(\cup_{C \in \fC_{\bar y}(\fH_{V_y,q})}\bar C) \cap V_{\fH_{V_y,q}}^{subreg}$ is connected. Hence by Lemma \ref{key11}, $A'$ and $A''$ are in the same connected component of $(\cup_{B \in \fA_y} \bar B) \cap V^{subreg}$.

By Lemma \ref{sequence}, there is a sequence of alcoves $A=A_0, A_1, \cdots A_r=A''$ in $C$ such that $\bar A_i$ contains a regular point of $V_y$ and $\bar A_{i-1} \cap \bar A_i \cap V_y$ spans a codimension one affine subspace $P_i=H_i \cap V_y$ of $V_y$ for $1 \le i \le r$.

If $P_i=yP_i$. Then $H_i \in \fH(\nu_y)$. By condition (2), $s_{H_i}(V_y)=V_y$. Since $V_y \nsubseteq H_i$, $H_i$ is the affine hyperplane containing $P_i$ and orthogonal to $V_y$. Hence $H_i$ is the unique element in $\fH$ whose intersection with $V_y$ is $P_i$ and thus the unique hyperplane separating $A_{i-1}$ from $A_i$. So $A_{i-1}$ and $A_i$ are in the same connected component of $(\cup_{B \in \fA_y}\bar B) \cap V^{subreg}$.

If $y P_i \neq P_i$, then there is a sequence of alcoves $A_{i-1}=B_0, B_1, \cdots, B_s=A_i$ in $C$ such that $B_{k-1}$ and $B_k$ have a common face and $\bar B_{k-1} \cap \bar B_k \cap V_y$ spans $P_i$ for $k=1, \cdots, s$. By Proposition \ref{con}, we see that all $B_k \in \fA_y$. Hence $A_{i-1}$ and $A_i$ are in the same connected component of $(\cup_{B \in \fA_y}\bar B) \cap V^{subreg}$.

Hence $A$ and $A''$ are in the same connected component of $(\cup_{B \in \fA_y}\bar B) \cap V^{subreg}$.
\end{proof}

\begin{cor}
Let $y \in \tW$ such that $\bar y$ is an elliptic element in $\tW_0$. Then the $W$-conjugacy class of $y$ is nice.
\end{cor}

\

Now we classify straight nice conjugacy classes.

\begin{prop}
Let $\co$ be a straight $W$-conjugacy class of $\tW$ and $\co \subset \tW_G \s$ with $\s \in \Om'$. Then $\co$ is nice if and only if there exists $x \in \co$ such that $\nu_x=\nu_\co$ and $x$ is superbasic in $\tW_J$, where $J \subset J_{\co}$ is a union of connected components of Dynkin diagram of $J_{\co}$ and $\s$ fixes each element of $J_\co-J$.
\end{prop}
\begin{proof}
Assume that $x \in \co$ such that $\nu_x=\nu_\co$ and $x$ is superbasic in $\tW_J$, where $J \subset J_{\co}$ is a union of connected components of Dynkin diagram of $J_{\co}$ and $\s$ fixes each element of $J_\co-J$. Let $H=H_{\a, k} \in \fH(\nu_x)$. Then $\<\nu_x, \a\>=0$ and $\a$ is a linear combination of roots in $J_\co$. If $\a$ is a linear combination of roots in $J$, then by Proposition \ref{superstraight}, $H \cap V_x=\emptyset$. If $\a$ is a linear combination of roots in $J_\co-J$, then $s_H (V_x)=V_x$.  By Theorem \ref{affinenice}, $\co$ is nice.

Assume that $\co$ is nice. By Proposition \ref{cri} and Theorem \ref{affinenice}, there exists a basic element $x \in \tW_{J_\co}$ such that $\nu_x=\nu_\co$ and $\bar x$ is weakly elliptic in $W_{J_\co} \rtimes \<\s\>$. Set $J=\supp(\bar x)$. Then $J$ is a union of connected components of Dynkin diagram of $J_\co$ and $\s$ fixes each element of $J_\co-J$. Let $H=H_{\a,k}$ such that $\a$ is a linear combination of roots in $J$. Since $V^{\bar x} \subset V^{W_J}$, we have that $V^{\bar x} \subset H_{\a, 0}$. By Proposition \ref{cri}, $V_x \nsubseteq H$. Notice that $V_x$ is parallel to $V^{\bar x}$. Thus $V_x \cap H=\emptyset$. By Proposition \ref{superstraight}, $x$ is superbasic in $\tW_J$.
\end{proof}

\begin{cor}
Let $\co$ be a superstraight $W$-conjugacy class of $\tW$. Then $\co$ is nice.
\end{cor}

\section{Class polynomial}

\subsection{} Set $\ca=\ZZ[v, v \i]$. The Hecke algebra $H$ associated to $W$ is the associated $\ca$-algebra with basis $T_w$ for $w \in W$ and the multiplication is given by
\begin{gather*} T_x T_y=T_{x y}, \quad \text{ if } \ell(x)+\ell(y)=\ell(x y); \\ (T_s-v)(T_s+v \i)=0, \quad \text{ for } s \in S. \end{gather*}

Then $T_s \i=T_s-(v-v \i)$ and $T_w$ is invertible in $H$ for all $w \in W$. If $\d$ is an automorphism of $W$ with $\d(S)=S$, then $T_w\mapsto T_{\d(w)}$ induces an $A$-linear automorphism of $H$ which is still denoted by $\d$. The Hecke algebra associated to $\tW=W \rtimes \Om$ is defined to be $\tH=H\rtimes \Om$. It is easy to see that $\tH$ is the associated $\ca$-algebra with basis $T_w$ for $w \in \tW$ and multiplication is given by
\begin{gather*} T_{\tx} T_{\ty}=T_{\tx \ty}, \quad \text{ if } \ell(\tx)+\ell(\ty)=\ell(\tx \ty); \\ (T_s-v)(T_s+v \i)=0, \quad \text{ for } s \in S. \end{gather*}

Let $h, h' \in \tH$, we call $[h, h']=h h'-h'h$ the {\it commutator} of $h$ and $h'$. Let $[\tH, \tH]$ be the $\ca$-submodule of $\tH$ generated by all commutators.

For finite Hecke algebras, Geck and Pfeiffer introduced class polynomials in \cite{GP}. We'll show in the section that their construction can be generalized to affine Hecke algebra based on Theorem \ref{main}.

\begin{lem}\label{e}
Let $\tw, \tw' \in \tW$ with $\tw \tilde \sim \tw'$. Then $$T_{\tw} \equiv T_{\tw'} \mod [\tH, \tH].$$
\end{lem}

Proof. By definition of $\tilde \sim$, it suffices to prove the case that there exists $x \in \tW$ such that $\tw'=x \tw x \i$ with $\ell(\tw)=\ell(\tw')$ and either $\ell(x \tw)=\ell(x)+\ell(\tw)$ or $\ell(\tw x \i)=\ell(x)+\ell(\tw)$.

If $\ell(x \tw)=\ell(x)+\ell(\tw)$, then $\ell(\tw' x)=\ell(x \tw)=\ell(\tw')+\ell(\d(x))$. Hence $T_{\tw'} T_{x}=T_{\tw' x}=T_{x \tw}=T_{x} T_{\tw}$ and $T_{\tw'}=T_{x} T_{\tw} T_{x} \i \equiv T_{\tw} \mod [\tH, \tH]$.

If $\ell(\tw x \i)=\ell(x)+\ell(\tw)$, then $\ell(x \i \tw')=\ell(\tw x \i)=\ell(\tw')+\ell(x)$. Hence $T_{x \i} T_{\tw'}=T_{x \i \tw'}=T_{\tw x \i}=T_{\tw} T_{x \i}$ and $T_{\tw'}=T_{x \i} \i T_{\tw} T_{x \i} \equiv [\tH, \tH]$. \qed

\

Now Corollary \ref{maincor} (2) implies that

\begin{cor}
Let $\co$ be a conjugacy class of $\tW$ and $\tw, \tw' \in \co_{\min}$. Then $$T_{\tw} \equiv T_{\tw'} \mod [\tH, \tH].$$
\end{cor}

\begin{thm}\label{class}
Let $\tw \in \tW$. Then for any conjugacy class $\co$ of $\tW$, there exists a polynomial $f_{\tw, \co} \in \ZZ[v-v \i]$ with nonnegative coefficient such that $f_{\tw, \co}$ is nonzero only for finitely many $\co$ and \[\tag{a} T_{\tw} \equiv \sum_{\co} f_{\tw, \co} T_{\tw_{\co}} \mod [\tH, \tH],\] where $\tw_\co$ is a minimal length element for each conjugacy class $\co$ of $\tW$.
\end{thm}

Proof. We argue by induction on $\ell(\tw)$.

If $\tw$ is a minimal element in a conjugacy class of $\tW$, then we set $f_{\tw, \co}=\begin{cases} 1, & \text{ if } \tw \in \co \\ 0, & \text{ if } \tw \notin \co \end{cases}$. In this case, the statement automatically holds.

Now we assume that $\tw$ is not a minimal element in the conjugacy class of $\tW$ that contains it and that for any $\tw' \in \tW$ with $\ell(\tw')<\ell(\tw)$, the statement holds for $\tw'$. By Theorem \ref{main}, there exists $\tw_1 \approx \tw$ and $i \in S$ such that $\ell(s_i \tw_1 s_i)<\ell(\tw_1)=\ell(\tw)$. In this case, $\ell(s_i \tw)<\ell(\tw)$ and we define $f_{\tw, \co}$ as $$f_{\tw, \co}=(v-v \i) f_{s_i \tw_1, \co}+f_{s_i \tw_1 s_i, \co}.$$

By inductive hypothesis, $f_{s_i \tw_1, \co}, f_{s_i \tw_1 s_i, \co} \in \ZZ[v-v \i]$ with nonnegative coefficients. Hence $f_{\tw, \co} \in \ZZ[v-v \i]$ with nonnegative coefficients. Moreover, there are only finitely many $\co$ such that $f_{s_i \tw_1, \co} \neq 0$ or $f_{s_i \tw_1 s_i, \co} \neq 0$. Hence there are only finitely many $\co$ such that $f_{\tw, \co} \neq 0$.

By Lemma \ref{e}, $T_{\tw} \equiv T_{\tw_1} \mod [\tH, \tH]$. Now

\begin{align*}
T_{\tw} & \equiv T_{\tw_1}=T_{s_i} T_{s_i \tw_1 s_i} T_{s_i} \equiv T_{s_i \tw_1 s_i} T_{s_i} T_{s_i} \\ &=(v-v \i) T_{s_i \tw_1 s_i} T_{s_i}+T_{s_i \tw_1 s_i}=(v-v \i) T_{s_i \tw_1}+T_{s_i \tw_1 s_i} \\ & \equiv (v-v \i) \sum_\co f_{s_i \tw_1, \co} T_{\tw_{\co}}+\sum_\co f_{s_i \tw_1 s_i, \co} T_{\tw_{\co}} \\ &=\sum_\co f_{\tw, \co} T_{\tw_\co} \mod [\tH, \tH].
\end{align*}

Hence the statement holds for $\tw$. \qed

\subsection{}\label{c} We call $f_{\tw, \co}$ in the above theorem the {\it class polynomials}.

Notice that the construction of $f_{\tw, \co}$ depends on the choice of the sequence of elements in $S$ used to conjugate $\tw$ to a minimal length element in its conjugacy class. We'll show in the next section that $f_{\tw, \co}$ is in fact, independent of such choice and is uniquely determined by the identity (a) in Theorem \ref{class}. For a finite Hecke algebra, similar result was obtained by Geck and Rouquier in \cite[Theorem 4.2]{GR}.

\section{Cocenter of $\tH$}

\subsection{} We call $\tH/[\tH, \tH]$ the {\it cocenter} of $\tH$. Now for any conjugacy class $\co$ of $\tW$, we choose a minimal length representative $\tw_\co$. Then by Theorem \ref{class}, the image of $T_{\tw_\co}$ (for all conjugacy classes $\co$ of $\tW$) spans the cocenter of $\tH$. The main purpose of this section is to show that the images of $T_{\tw_\co}$ are linearly independent in the cocenter of $\tH$ and thus form a basis. We call it the {\it standard basis} of $\tH/[\tH, \tH]$.


\subsection{} Following \cite{L2}, let $J$ be the based ring of $W$ with basis $(t_w)_{w \in W}$. For each $\d\in\Om$, the map $t_w \mapsto t_{\d(w)}$ gives a ring homomorphism of $J$, which we still denote by $\d$. Set $\tJ=J \rtimes \Om$, $J_\ca=J \otimes_\ZZ \ca$ and $\tJ_\ca=\tJ \otimes_\ZZ \ca$.

Let $(c_w)_{w \in W}$ be the Kazhdan-Lusztig basis of $H$. Then $c_x c_y=\sum_{z \in W} h_{x, y, z} c_z$ with $h_{x, y, z} \in \ca$. There is a homomorphisms of $\ca$-algebras $\varphi: H \to J_\ca$ defined by
$c_w \mapsto\sum_{d\in\cd,a(d)=a(x)}h_{w,d,x}t_{x}$,
where $a:\tW\rightarrow\NN$ is Lusztig's $a$-function and $\cd$ is the set of distinguished involutions of $W$. It is easy to see that for each $\d\in\Om$ $$\varphi(\d(c_w))=\varphi(c_{\d(w)})=\sum_{d \in \cd, a(d)=a(x)} h_{\d(w), \d(d), \d(x)} t_{\d(x)}=\d \varphi(c_w).$$ Hence $\varphi$ extends in a natural way to a homomorphism $\tilde \varphi: \tH \to \tJ_\ca$.

\subsection{}\label{Psi}
Set $J_\CC=J \otimes_\ZZ \CC$ and $\tJ_\CC=\tJ \otimes_\ZZ \CC$. Set $H_\CC=H \otimes_{\ca} \CC[v, v \i]$ and $\tH_\CC=\tH \otimes_{\ca} \CC[v, v \i]$. For any $q \in \CC^\times$, let $\CC_q$ be the one-dimemsion $\ca$-module over $\CC$ such that $v$ acts as scalar $q$. Set $H_q=H \otimes_{\ca} \CC_q$ and $\tH_q=\tH \otimes_{\ca} \CC_q$.

Let $E$ be a $\tJ_\CC$-module. Through the homomorphism $$\tilde \varphi_q=\tilde \varphi \mid_{v=q}: \tH_q \to \tJ_\CC,$$ it is endowed with an $\tH_q$-module structure. We denote this $\tH_q$-module by $E_q$.

For an associative $\CC$-algebra $R$, we denote by $\K(R)$ the Grothendieck group of all finite dimensional representations of $R$ over $\CC$.

Thus the map $E \mapsto E_q$ induced a map $(\tilde \varphi_q)_*: \K(\tJ_\CC) \to \K(\tH_q)$. We have that
\begin{lem}\label{Psi}
The map $(\tilde \varphi_q)_*: \K(\tJ_\CC) \to \K(\tH_q)$ is surjective.
\end{lem}

\begin{proof}
The proof is similar to \cite[Lemma 1.9]{L3}. Since $\Om$ preserves $a$-function, we associate to each $\tH_q$-module $M$ an integer $a_M$ such that $c_w M=0$ whenever $a(w)>a_M$ and $c_{w'} M\neq 0$ for some $w'\in W$ with $a(w')=a_M$. For each simple $\tH_q$-module $M$, we construct as in the proof of \cite[Lemma 1.9]{L3} a finite dimensional $\tJ_{\CC}$-module $\tM$ and a nonzero $\tH_q$-morphism $p:\tM_q \to M$ with $M'=\ker p$ such that $a_{M'}<a_M$. Since the $a$-function is bounded, it follows easily that $(\tilde \varphi_q)_*$ is surjective.
\end{proof}

\subsection{}\label{Psi2} Let $\K^*(\tH_q)=\Hom_\ZZ(\K(\tH_q), \CC)$ be the space of linear functions on $\K(\tH_q)$. The map $\tH_q \to \K^*(\tH_q)$ sending $h \in \tH_q$ to the function $M \mapsto Tr_M(h)$ on $\K(\tH_q)$ induces a map $$\Psi_q: \tH_q/[\tH_q, \tH_q] \to \K^*(\tH_q).$$

It is proved in \cite[Theorem B]{K} and \cite[Main Theorem]{F} that $\Psi_q$ is injective if $q$ is a power of prime. We'll show below that $\Psi_1$ is also injective. Here $\tH_1=\CC[\tW]$ is the group algebra.

\begin{lem}
Let $\tw, \tw' \in \tW$. If $\tw$ and $\tw'$ are not in the same conjugacy class of $\tW$, then there exists $n>0$ such that the image of $\tw$ and $\tw'$ in $\tW/n P$ are not in the same conjugacy class.
\end{lem}

\begin{proof}
It suffices to consider the case where $\tW \subset W \rtimes \Aut(W, S)$. Then we have that $\tW/W$ is finite. Notice that $Q$ is a normal subgroup of $\tW$ and $\tW/m Q$ is a quotient group of $\tW$. Since $\tW/W$ and $W/Q$ are both finite groups, $\tW/Q$ is also a finite group. We choose a representative $x_i$ for each coset of $Q$. Set $\tw_i=x_i \tw x_i \i$. Then any element conjugate to $\tw$ is of the form $t^\l \tw_i t^{-\l}$ for some $i$ and $\l \in Q$.

Now we show that

(a) for any $i$, there exists $n_i>0$ such that the image of $\tw_i$ and $\tw'$ are not conjugate by an element in $n_i Q$.

Otherwise, there exists $i$ such that $\tw'=t^\chi \tw_i$ for some $\chi \in Q$ such that $\chi \in (1-\tw_i) Q+n Q$ for all $n>0$. Therefore the image of $\chi$ in $Q/(1-\tw_i)Q$ is divisible by all the positive integer $n$. So the image of $\chi$ in $Q/(1-\tw_i)Q$ is $0$ and $\chi=\l-\tw_i \l$ for some $\l \in Q$. Hence $\tw'=t^{\l} \tw_i t^{-\l}$ is conjugate to $\tw$ in $\tW$. That is a contradiction.

(a) is proved.

Now set $n=\Pi_i n_i$. If the image of $\tw$ and $\tw'$ in $\tW/n Q$ are in the same conjugacy class, then there exists $i$ such that the image of $\tw_i$ and $\tw'$ in $\tW/n Q$ are conjugate by an element in $Q$. Hence their image in $\tW/n Q \to \tW/  n_i Q$ are conjugate by an element in $Q$. That contradicts (a).
\end{proof}

\begin{prop}\label{kl}
Let $\tw_1,\cdots,\tw_r\in\tW$ be elements in distinct conjugacy classes of $\tW$. The the elements $\Psi_1(\tw_1),\cdots,\Psi_1(\tw_r)$ are linearly independent functions on $\K(\CC[\tW])$.
\end{prop}
\begin{proof}
Again, it suffices to consider the case where $\tW/Q$ is finite.

For any $1 \le i<j \le r$, there exists $n_{i j}>0$ such that the image of $\tw_i$ and $\tw_j$ in $\tW/n_{i j} Q$ are not in the same conjugacy class.

Set $n=\Pi_{1 \le i<j \le r} n_{i j}$ and $F=\tW/n P$. Then $F$ is a finite group. Let $\underline \tw_i$ be the image of $\tw_i$ in $F$. If $\underline \tw_i$ and $\underline \tw_j$ are in the same conjugacy class of $F$, then their image in $\tW/n_{i j} P$ under the map $F \to \tW/n_{i j} P$ are still in the same conjugacy class. That is a contradiction. Hence $\underline \tw_1, \cdots, \underline \tw_r$ are in distinct conjugacy classes of $F$.

The surjection $\tW \to F$ induces an injection $\K(\CC[F]) \to \K(\CC[\tW])$ and a surjection $\K^*(\CC[\tW]) \to \K^*(\CC[F])$. We have the following commutative diagram
\[\xymatrix{\tW \ar[r]^-{\Psi_1} \ar[d] & \K^*(\CC[\tW]) \ar[d] \\ F \ar[r]^-{\Psi} & \K^*(\CC[F])}\]
Here $\Psi: F \to \K^*(\CC[F])$ is defined in the same way as $\Psi_q$ in $\S$\ref{Psi2}.

Since $F$ is a finite group and $\underline \tw_1, \cdots, \underline \tw_r$ are in distinct conjugacy classes of $F$, $\Psi(\underline\tw_1), \cdots, \Psi(\underline \tw_r)$ are linearly independent functions on $\K(\CC[F])$. Hence $\Psi_1(\tw_1) \cdots, \Psi_1(\tw_r)$ are linearly independent functions on $\K(\CC[\tW])$.
\end{proof}



\begin{cor}
The map $\Psi_1: \CC[\tW]/[\CC[\tW], \CC[\tW]] \to \K^*(\CC[\tW])$ is injective.
\end{cor}

\

Now we prove the main result of this section.

\begin{thm}\label{indep}
Let $\tw_1,\cdots,\tw_r\in\tW$ be elements in distinct conjugacy classes of $\tW$. Then the image of $T_{\tw_1},\cdots,T_{\tw_r}$ in $\tH_\CC/[\tH_\CC, \tH_\CC]$ are linearly independent.
\end{thm}
\begin{proof}
Set $\ca_\CC=\CC[v, v\i]$. Assume that $\sum_{i=1}^rc_iT_{\tw_i}\in[\tH_\CC,\tH_\CC]$ with all $c_i\in\ca_\CC$. Suppose that not all $c_i$'s are $0$. Then there exists $c \in \ca_\CC$ and $c_i'\in\ca_\CC$ for $1 \le i \le r$ such that $c_i=cc_i'$ and there exists $1 \le s \le r$ with $c'_s \mid_{v=1} \neq 0$.

Let $E$ be a $\tJ$-module over $\CC$. Set $E_\CC=E\otimes_\CC \ca_\CC$. Then $E_\CC$ is a $\tJ_{\ca_\CC}$-module over $\ca_\CC$. Via the homomorphism $\tilde \varphi: \tH \to \tJ_\ca$, $E_\CC$ admits an $\tH_{\CC}$-module structure over $\ca_\CC$. We denote it by $E_\varphi$. For $h\in\tH$, let $\tr_E(h)\in\ca_{\CC}$ be the trace of the endomorphism $h$ on $E_\varphi$ over $\ca_{\CC}$. Then $\tr_E(\sum_{i=1}^rc_iT_{\tw_i})=c(\sum_{i=1}^rc_i'\tr_E(T_{\tw_i}))=0$. Note that $\ca_{\CC}$ is an integral domain and $c\neq0$, we have that $\sum_{i=1}^r c_i'\tr_E(T_{\tw_i})=0$ for all $\tJ$-modules $E$.
Set $v=1$, then \[\tag{a} \sum_{i=1}^r c_i' \mid_{v=1} \tr_E(T_{\tw_i})=0\] for all $\tJ$-module $E$. Hence by Lemma \ref{Psi}, (a) holds for all $\tW$-modules. Now by Proposition \ref{kl},, $c'_i \mid_{v=1}=0$ for all $i$. That is a contradiction.
\end{proof}

\begin{cor}
The class polynomial $f_{\tw, \co}$ is uniquely determined by the identity (a) in Theorem \ref{class}.
\end{cor}

\

Now combining Theorem \ref{class} and Theorem \ref{indep}, we have that

\begin{thm}
For any conjugacy class $\co$ of $\tW$, we choose a minimal length representative $\tw_\co$. Then $(T_{\tw_\co})_{\co}$ is a $\ca$-basis of $\tH/[\tH, \tH]$.
\end{thm}

\section*{Acknowledgement} We thank George Lusztig, Weiqiang Wang, Nanhua Xi and Xinwen Zhu for helpful discussions on affine Hecke algebras.

\end{document}